\documentclass[12pt,reqno]{amsart}
\usepackage{amsmath,amssymb,extarrows}
\usepackage{url}
\usepackage{tikz,enumerate}
\usepackage{diagbox}
\usepackage{appendix}
\usepackage{epic}

\usepackage{float} 

\usepackage{cite}

\usepackage{hyperref}
\usepackage{array}

\usepackage{booktabs}

\allowdisplaybreaks[4]

\newtheorem{theorem}{Theorem}
\newtheorem{corollary}[theorem]{Corollary}

\newtheorem{lemma}[theorem]{Lemma}

\theoremstyle{definition}

\theoremstyle{remark}
\newtheorem{rem}{Remark}
\numberwithin{equation}{section}
\numberwithin{theorem}{section}
\numberwithin{defn}{section}

\DeclareMathOperator{\diag}{diag}

\begin{document}
\title[Mizuno's Rank Three Nahm Sums \uppercase\expandafter{\romannumeral1}]{Mizuno's rank three Nahm sums \uppercase\expandafter{\romannumeral1}: identities of index $(1,1,2)$}

\author{Boxue Wang and Liuquan Wang}
\address[B. Wang]{School of Mathematics and Statistics, Wuhan University, Wuhan 430072, Hubei, People's Republic of China}
\email{boxwang@whu.edu.cn}
\address[L. Wang]{School of Mathematics and Statistics, Wuhan University, Wuhan 430072, Hubei, People's Republic of China}
\email{wanglq@whu.edu.cn;mathlqwang@163.com}



\begin{abstract}
Mizuno provided 19 examples of generalized rank three Nahm sums with symmetrizer $\diag(1,1,2)$ which are conjecturally modular. We confirm their modularity by establishing Rogers--Ramanujan type identities of index $(1,1,2)$ for these examples. We first reduce these Nahm sums to some double sums or single sums, and then we use known results or apply the theory of Bailey pairs to prove the desired identities. Meanwhile, we generalize some triple sum identities to general multi-sum identities.
\end{abstract}

\maketitle

\section{Introduction}\label{sec-intro}
The famous Rogers--Ramanujan identities assert that
\begin{align}
    \sum_{n=0}^\infty\frac{q^{n^2}}{(1-q)(1-q^2)\cdots (1-q^n)}
    =\prod\limits_{n=0}^\infty \frac{1}{(1-q^{5n+1})(1-q^{5n+4})}, \label{R.R.1}
    \\
    \sum_{n=0}^\infty\frac{q^{n^2+n}}{(1-q)(1-q^2)\cdots (1-q^n)}
    =\prod\limits_{n=0}^\infty \frac{1}{(1-q^{5n+2})(1-q^{5n+3})}. \label{R.R.2}
\end{align}
These identities were first proved by Rogers \cite{Rogers1894} and later rediscovered by Ramanujan. They inspired lots of works on finding similar sum-product identities which were usually called as Rogers-Ramanujan type identities.

For convenience, from now on we shall use standard $q$-series notation:
\begin{align}
   & (a;q)_\infty:=\prod\limits_{k=0}^\infty (1-aq^k), \quad |q|<1, \label{aq-infinite-defn}\\
   & (a;q)_n:=\frac{(a;q)_\infty}{(aq^n;q)_\infty}, \quad n\in \mathbb{R}. \label{aq-defn}
\end{align}
We will frequently use compressed notation:
\begin{align}
    (a_1,a_2,\dots,a_m;q)_n:=(a_1;q)_n(a_2;q)_n\cdots (a_m;q)_n, \quad n \in \mathbb{N}\cup \{\infty\}.
\end{align}

As a multi-sum generalization of \eqref{R.R.1} and \eqref{R.R.2}, the Andrews--Gordon identity (see \cite{Andrews1974,Gordon1961}) states that for integer $k>2$ and $1\leq i \leq k$,
\begin{align}
&\sum_{n_1,\dots,n_{k-1}\geq 0} \frac{q^{N_1^2+\cdots+N_{k-1}^2+N_i+\cdots +N_{k-1}}}{(q;q)_{n_1}\cdots (q;q)_{n_{k-2}} (q;q)_{n_{k-1}}} &=\frac{(q^i,q^{2k+1-i},q^{2k+1};q^{2k+1})_\infty}{(q;q)_\infty},\label{AG}
\end{align}
where $N_j=n_j+\cdots+n_{k-1}$ if $j\leq k-1$ and $N_k=0$.  As the even moduli companion of it, Bressoud's identity \cite{Bressoud1979} asserts that for integers $k\geq 2$ and $1\leq i \leq k$,
\begin{align}\label{eq-Bressoud}
\sum_{n_1,\dots,n_{k-1}\geq 0} \frac{q^{N_1^2+\cdots+N_{k-1}^2+N_i+\cdots +N_{k-1}}}{(q;q)_{n_1}\cdots (q;q)_{n_{k-2}} (q^2;q^2)_{n_{k-1}}} =\frac{(q^i,q^{2k-i},q^{2k};q^{2k})_\infty}{(q;q)_\infty}
\end{align}
where $N_j$ is defined as before.

For convenience, we shall adopt a notion from \cite{Wang2021}. Let $n_1,\ldots,n_k$ be positive integers whose greatest common divisor equals 1. Let $t(i_1,\ldots,i_k)$ denote some integer-valued functions, and let $Q(i_1,\dots,i_k)$ be a rational polynomial in variables $i_1,\dots,i_k$. If the sum side of a Rogers--Ramanujan type identity is a mixed sum of
$$\sum_{(i_1,\dots,i_k)\in \mathbb{N}^k}\frac{(-1)^{t(i_1,\dots,i_k)}q^{Q(i_1,\dots,i_k)}}{(q^{n_1};q^{n_1})_{i_1}\cdots (q^{n_k};q^{n_k})_{i_k}},$$
then we call it as an identity of index $(n_1,\dots,n_k)$. Using this notion, we see that the Andrews--Gordon identity is of index $(1,1,\dots,1)$ and the Bressoud identity is of index $(1,\dots, 1,2)$. In the past decade, a number of multi-sum Rogers--Ramanujan type identities have been discovered. For example, Kanade and Russell \cite{KR2015,KR2019} conjectured some identities of indexes $(1,2,3)$ and $(1,4,6)$. Most of them have been confirmed by Bringmann, Jennings-Shaffer and Mahlburg \cite{BJM} and Rosengren \cite{Rosengren}. Cao and Wang \cite{CW} proved some identities of indexes
$$(1, 1),(1, 2),(1, 1, 1),(1, 1, 2),(1, 1, 3),(1, 2, 2),(1, 2, 3),(1, 2, 4).$$
Li and Wang \cite{LW} found many identities of indexes
$$(1,1),(1,2),(1,4),(1,1,1),(1,1,3),(1,2,2),(1,2,4),(1,1,1,2).$$
See the references therein for more related works.

Rogers--Ramanujan type identities have important implications in various areas of mathematics such as combinatorics, Lie theory and number theory. We refer the reader to Sills' book \cite{Sills-book} for more detailed introduction. Here we focus on the relations between them and modular forms.
As is well known, these identities serve as one of the bridges linking the theory of $q$-series and modular forms. The product sides of \eqref{R.R.1} and \eqref{R.R.2} are clearly modular forms, but this fact is not easy to be observed from the sum sides. A key problem arises here: what kinds of $q$-hypergeometric series are modular forms?  In this aspect, Nahm \cite{Nahm1994,Nahmconf,Nahm2007} considered a particular class of $q$-hypergeometric series:
$$f_{A,b,c}(q):=\sum_{n=(n_1,\dots,n_r)^\mathrm{T} \in (\mathbb{Z}_{\geq 0})^r} \frac{q^{\frac{1}{2}n^\mathrm{T} An+n^\mathrm{T} b+c}}{(q;q)_{n_1}\cdots (q;q)_{n_r}},$$
where $r\geq 1$ is a positive integer, $A$ is a real positive definite symmetric $r\times r$ matrix, $b$ is a vector of length $r$, and $c$ is a scalar. The series $f_{A,b,c}(q)$ is later referred as Nahm series. Nahm's problem is to determine all $(A,b,c)$ with rational entries such that $f_{A,b,c}(q)$ is modular. For convention, we shall call such $(A,b,c)$ as a rank $r$ modular triple and call $A$ as its matrix part.

Nahm  made a conjecture which provides a sufficient and necessary condition on $A$ so that it is the matrix part of a modular triple. This conjecture is known to be false in general but one direction has been confirmed recently by works of Calegari, Garoufalidis and Zagier \cite{CGZ}. In 2007, Zagier systematically investigated Nahm's problem. He proved that there are exactly seven modular triples in the rank one case. Zagier also provided two lists of possible modular triples in the rank two and rank three cases, respectively. The modularity of these triples have now all been confirmed by works of Zagier \cite{Zagier}, Vlasenko and Zwegers \cite{VZ}, Cherednik and Feigin \cite{Feigin}, Wang \cite{Wang-rank2,Wang-rank3}, Cao, Rosengren and Wang \cite{CRW}, etc.

In 2023, Mizuno considered a more general class of $q$-hypergeometric series which we shall still call as (generalized) Nahm sums.  Suppose $d=(d_1,\dots,d_r)\in \mathbb{Z}_{>0}^r$,   $b \in \mathbb{Q}^r$ is a vector and  $c \in \mathbb{Q}$ is a scalar. We call $A \in \mathbb{Q}^{r \times r} $ a symmetrizable matrix with the symmetrizer $D := \mathrm{diag}(d_1,\dots, d_r)$ if $AD$ is symmetric positive definite. Mizuno defined   \begin{align}
   \widetilde{f}_{A,b,c,d}(q):= \sum_{n=(n_1,\dots,n_r)^\mathrm{T} \in (\mathbb{Z}_{\geq 0})^r} \frac{q^{\frac{1}{2}n^\mathrm{T}ADn+n^\mathrm{T}b+c}}{(q^{d_1};q^{d_1})_{n_1}\cdots (q^{d_r};q^{d_r})_{n_r}}.
\end{align}
In particular, when $d_1=d_2=\cdots=d_r=1$, this is the original series considered by Nahm. As before, we shall call $(A,b,c,d)$ as a rank $r$ modular quadruple if $f_{A,b,c,d}(q)$ is modular.

Following the numerical method used by Zagier \cite{Zagier}, Mizuno \cite{Mizuno} searched for possible modular quadruples of rank 2 or 3. In the rank two case, Mizuno \cite[Table 1]{Mizuno} provided a list of 14 possible modular quadruples and verified the modularity of four of them. The authors \cite{WW} confirmed the modularity of eight  modular quadruples and provided conjectural identities for the remaining two examples.

In Tables 2 and 3 of \cite{Mizuno}, Mizuno provided 19 and 15 possible rank three modular quadruples associated with the symmetrizer $d=(1,1,2)$ and $d=(1,2,2)$, respectively. In this sequel of papers, we will discuss these examples and confirm their modularity. In particular, this paper discuss the examples in Table 2. We will prove that all the Nahm sums associated with the quadruples in \cite[Table 2]{Mizuno} are indeed modular. The examples in Table 3 will be discussed in a separate paper.

For convenience, we label the examples from 1 to 19 according to their order of appearances in \cite[Table 2]{Mizuno}. Each example consists of several quadruples with the common matrix part. To prove the modularity of the corresponding Nahm sums, we will establish Rogers--Ramanujan type identities for them. That is, for those quadruples $(A,b,c,d)$ in Table 2, we establish identities of index $(1,1,2)$ for each of the series $f_{A,b,c,d}(q)$. The product sides of these identities clearly reveal the modularity of these Nahm sums. For instance, Example 13 of Table 2 corresponds to
\begin{align*}
A=\begin{pmatrix}
1 & 0 & 1/2\\
0 & 2 & 1\\
1 & 2 & 2
\end{pmatrix},
\quad
b\in  \bigg\{
\begin{pmatrix}
0 \\ 0 \\ 0
\end{pmatrix},
\begin{pmatrix}
0 \\ 0 \\ 1
\end{pmatrix},
\begin{pmatrix}
0 \\ 1 \\ 2
\end{pmatrix},
\begin{pmatrix}
1 \\ 1 \\ 2
\end{pmatrix}
\bigg\}.
\end{align*}
We find the modular product expressions for the corresponding Nahm sums:
\begin{align}
    \sum_{i,j,k\ge 0}\frac{q^{i^2+2j^2+4k^2+2ik+4jk}}{(q^2;q^2)_i(q^2;q^2)_j(q^4;q^4)_k}
    &=
    \frac{(q^5,q^6,q^{11};q^{11})_\infty}{(q;q^2)_\infty(q^4;q^4)_\infty},\label{intro-table2.13.1}
    \\
    \sum_{i,j,k\ge 0}\frac{q^{i^2+2j^2+4k^2+2ik+4jk+2k}}{(q^2;q^2)_i(q^2;q^2)_j(q^4;q^4)_k}
    &=
    \frac{(q^4,q^7,q^{11};q^{11})_\infty}{(q;q^2)_\infty(q^4;q^4)_\infty},\label{intro-table2.13.2}
    \\\sum_{i,j,k\ge 0}\frac{q^{i^2+2j^2+4k^2+2ik+4jk+2j+4k}}{(q^2;q^2)_i(q^2;q^2)_j(q^4;q^4)_k}
    &=
    \frac{(q^2,q^9,q^{11};q^{11})_\infty}{(q;q^2)_\infty(q^4;q^4)_\infty},\label{intro-table2.13.3}
    \\\sum_{i,j,k\ge 0}\frac{q^{i^2+2j^2+4k^2+2ik+4jk+2i+2j+4k}}{(q^2;q^2)_i(q^2;q^2)_j(q^4;q^4)_k}
    &=
    \frac{(q,q^{10},q^{11};q^{11})_\infty}{(q;q^2)_\infty(q^4;q^4)_\infty}.\label{intro-table2.13.4}
\end{align}
The right side of these identities indicates that there might exist an identity with the numerator $(q^3,q^8,q^{11};q^{11})_\infty$. We find an identity for this missing case:
\begin{align}
    \sum_{i,j,k\ge 0}\frac{q^{i^2+2j^2+4k^2+2ik+4jk+2j+2k}(1+q^{2i+2j+4k+2})}{(q^2;q^2)_i(q^2;q^2)_j(q^4;q^4)_k}
    =
    \frac{(q^3,q^8,q^{11};q^{11})_\infty}{(q;q^2)_\infty(q^4;q^4)_\infty}.
    \end{align}

For most of the examples, we first establish the corresponding identities by reducing triple sums to double or single sums, and then we employ some known Rogers--Ramanujan type identities. In fact, except for Examples 11, 12, 13 and 15, all other 15 examples can be reduced to the rank two examples considered by Zagier \cite{Zagier}, and then we can use identities from the work of Wang \cite{Wang-rank2}. Example 11 corresponds to Bressoud's identity \eqref{eq-Bressoud} and its companion identity (see \eqref{eq-Bailey-general-id}). To establish identities for Examples 12 and 13, we use the machinery of Bailey pairs. Example 15 can be reduced to the Rogers--Ramanujan identities \eqref{R.R.1}--\eqref{R.R.2}.

Meanwhile, we find generalizations for many identities involved in Mizuno's examples. The key step in such generalization is to find some multi-sum Rogers--Ramanujan type identities. Some of them are known but occasionally we obtain some apparently new results. For example, to generalize the identities \eqref{intro-table2.13.1}--\eqref{intro-table2.13.3}, we first establish the following multi-sum identities.
\begin{theorem}\label{thm-gen-13-original}
Let $N_i=n_i+n_{i+1}+\cdots+n_k$ ($1\leq i\leq k$) and $N_i=0$ for $i>k$. For $1\leq i \leq k+1$ we have
\begin{align}
\sum_{n_1,\dots,n_k\ge 0} \frac{q^{N_1^2+N_2^2+\cdots+N_k^2+N_i+N_{i+1}+\cdots+N_k}}{(q;q)_{n_1}\cdots(q;q)_{n_{k-1}}(-q^{\frac{1}{2}};q)_{n_k}(q^2;q^2)_{n_k}}
 =
    \frac{(q^{i},q^{\frac{3}{2}+2k-i},q^{\frac{3}{2}+2k};q^{\frac{3}{2}+2k})_\infty}{(q;q)_\infty}.
\end{align}
\end{theorem}
As a consequence of this theorem, we have the following identity (see Corollary \ref{cor-gen-13}) which includes \eqref{intro-table2.13.1}--\eqref{intro-table2.13.3} as special instances: For $k \in \mathbb{N}^+, 1 \le i \le k+1$,
\begin{align}
   & \sum_{m,n_1,\dots,n_k\ge 0}\frac{q^{\frac{1}{2}m^2+mn_k+(N_1^2+N_2^2+\cdots+N_k^2)+N_i+N_{i+1}+\cdots+N_k}}{(q;q)_{m}(q;q)_{n_1}\cdots(q;q)_{n_{k-1}}(q^2;q^2)_{n_k}} \nonumber \\
&    =
    \frac{(-q^{\frac{1}{2}};q)_\infty(q^{i},q^{\frac{3}{2}+2k-i},q^{\frac{3}{2}+2k};q^{\frac{3}{2}+2k})_\infty}{(q;q)_\infty}. \label{intro-ex13-general}
\end{align}
We also find a generalization of \eqref{intro-table2.13.4} based on the following theorem (see Corollary \ref{cor-gen-13-last}).
\begin{theorem}\label{thm-gen-13-odd}
Let $N_i=n_i+n_{i+1}+\cdots+n_k$ ($1\leq i\leq k$). We have
\begin{align}
\sum_{n_1,\dots,n_k\geq 0} \frac{q^{N_1^2+N_2^2+\cdots+N_k^2+N_1+\cdots+N_k}}{(q;q)_{n_1}\cdots(q;q)_{n_{k-1}}(-q^{\frac{1}{2}};q)_{n_k+1}(q^2;q^2)_{n_k}}
=\frac{(q^{\frac{1}{2}},q^{1+2k},q^{\frac{3}{2}+2k};q^{\frac{3}{2}+2k})_\infty}{(q;q)_\infty}.
\end{align}
\end{theorem}
There might be an identity similar to Theorem \ref{thm-gen-13-original} that generalizes Theorem \ref{thm-gen-13-odd}. But we are unable to find it at this stage.
We also find similar generalizations for identities in some other examples.

The paper is organized as follows. In Section \ref{sec-pre} we first collect some auxiliary identities which will be used in our proofs. Then we briefly discuss the theory of Bailey pairs and introduce some known and new Bailey pairs. In Section \ref{sec-exam} we discuss Mizuno's examples in \cite[Table 2]{Mizuno} one by one and present identities for them. Finally, in Section \ref{sec-general} we discuss multi-sum generalizations of many identities in Section \ref{sec-exam}. In particular, we will present proofs for Theorems \ref{thm-gen-13-original} and \ref{thm-gen-13-odd} as well as an identity of Warnaar.

\section{Preliminaries}\label{sec-pre}

\subsection{Auxiliary identities}
First, we need Euler's $q$-exponential identities \cite[Corollary 2.2]{Andrews}
\begin{align}\label{euler}
\sum_{n=0}^{\infty}\frac{z^n}{(q;q)_n}
=
\frac{1}{(z;q)_{\infty}}, \quad|z|<1, \quad
\sum_{n=0}^{\infty}\frac{q^{(^n_2)} z^n}{(q;q)_n}
=
(-z;q)_{\infty}
\end{align}
and the Jacobi triple product identity \cite[Theorem 2.8]{Andrews}
\begin{align}\label{Jacobi}
(q,z,q/z;q)_\infty=\sum_{n=-\infty}^\infty (-1)^nq^{\binom{n}{2}}z^n.
\end{align}

We establish two useful lemmas which play as a key role in simplifying multi-sums.
\begin{lemma}\label{lem-12}
We have
\begin{align}
\sum_{i,j\geq 0}\frac{u^{i+2j}q^{\binom{i}{2}}}{(q;q)_i(q^2;q^2)_j}
&=\frac{1}
{(u;q)_{\infty}},\label{eq-lem}
\end{align}
and as a consequence: for $n\geq 0$,
\begin{align}
\sum_{i+2j=n} \frac{q^{i(i-1)/2}}{(q;q)_i(q^2;q^2)_j}=\frac{1}{(q;q)_n}.
\end{align}
\end{lemma}
\begin{proof}
The first identity follows from the fact that
\begin{align*}
    \frac{1}{(u;q)_\infty}=\frac{(-u;q)_\infty}{(u^2;q^2)_\infty}=\sum_{i=0}^\infty \frac{u^iq^{i(i-1)/2}}{(q;q)_i} \sum_{j=0}^\infty \frac{u^{2j}}{(q^2;q^2)_j}.
\end{align*}
The second identity follows by comparing the coefficients of $u^n$ on both sides.
\end{proof}

\begin{lemma}\label{lem-13}
We have
\begin{align}
    \sum_{i,j\ge 0}\frac{u^{i+j}q^{i^2+j^2-i}}{(q^2;q^2)_i(q^2;q^2)_j}
    =
    (-u;q)_\infty,
\end{align}
and as a consequence: for $n\geq 0$,
\begin{align}
    \sum_{i+j=n}\frac{q^{i^2+j^2-i}}{(q^2;q^2)_i(q^2;q^2)_j}
    =\frac{q^{(n^2-n)/2}}{(q;q)_n}.
\end{align}
\end{lemma}
\begin{proof}
The first identity follows from the fact that
\begin{align*}
(-u;q)_\infty=(-u,-uq;q^2)_\infty=\sum_{i=0}^\infty\frac{u^iq^{i^2-i}}{(q^2;q^2)_i}\sum_{j\ge 0}\frac{u^jq^{j^2}}{(q^2;q^2)_j}.
\end{align*}
The second identity follows by comparing the coefficients of $u^n$ on both sides.
\end{proof}

\subsection{Bailey pairs}
 A pair of sequences $(\alpha_n(a;q),\beta_n(a;q))$ is called a Bailey pair relative to $a$ if for all $n\geq 0$,
 \begin{align}
     \beta_n(a;q)=\sum_{k=0}^n\frac{\alpha_k(a;q)}{(q;q)_{n-k}(aq;q)_{n+k}}.
 \end{align}

 \begin{lemma}[Bailey's Lemma]
Suppose that $(\alpha_n(a;q),\beta_n(a;q))$ is a Bailey pair relative to $a$. Then $(\alpha_n'(a;q),\beta_n'(a;q))$ is also a Bailey pair relative to $a$ where
 \begin{equation}
 \begin{split}
 \alpha_n'(a;q)&:=\frac{(\rho_1,\rho_2;q)_n(aq/\rho_1\rho_2)^n}{(aq/\rho_1,aq/\rho_2;q)_n}\alpha_n(a;q), \\ \beta_n'(a;q)&:=\sum_{r=0}^n\frac{(\rho_1,\rho_2;q)_r(aq/\rho_1\rho_2;q)_{n-r}(aq/\rho_1\rho_2)^r}{(aq/\rho_1,aq/\rho_2;q)_n(q;q)_{n-r}}\beta_r(a;q).
 \end{split}
 \end{equation}
 Equivalently, if $(\alpha_n(a;q),\beta_n(a;q))$ is a Bailey pair, then
\begin{align}\label{eq-Bailey-general-id}
&\frac{1}{(aq/\rho_1,aq/\rho_2;q)_n}\sum_{j=0}^n \frac{(\rho_1,\rho_2;q)_j(aq/\rho_1\rho_2;q)_{n-j}}{(q;q)_{n-j}}\Big(\frac{aq}{\rho_1\rho_2} \Big)^j\beta_j(a;q) \nonumber \\
&=\sum_{r=0}^n \frac{(\rho_1,\rho_2;q)_r}{(q;q)_{n-r}(aq;q)_{n+r}(aq/\rho_1,aq/\rho_2;q)_r}\Big(\frac{aq}{\rho_1\rho_2} \Big)^r \alpha_r(a;q).
\end{align}
 \end{lemma}
We now recall several limiting cases of this lemma, which have been widely used in the literature. If we let $\rho_2 \to \infty$, we obtain the Bailey pair:
\begin{align}
\begin{split}
   & \alpha_n'(a;q)=\frac{(\rho_1;q)_n(-aq/\rho_1)^nq^{\binom{n}{2}}}{(aq/\rho_1;q)_n}\alpha_n(a;q),\\
   & \beta_n'(a;q)=\sum_{r=0}^n\frac{(\rho_1;q)_r(-aq/\rho_1)^rq^{\binom{r}{2}}}{(aq/\rho_1;q)_r(q;q)_{n-r}}\beta_r(a;q).
\end{split}
\end{align}

If we further let $\rho_1\to \infty$ as well, we obtain the Baily pair \cite[Eq. (S1)]{BIS}:
\begin{align}
        \alpha'_n(a;q)=a^nq^{n^2}\alpha_n(a;q), \quad
        \beta'_n(a;q)=\sum_{r=0}^n\frac{a^rq^{r^2}}{(q;q)_{n-r}}\beta_r(a;q).  \tag{S1} \label{Bailey lemma-infty}
\end{align}
Here and below the labels (S1), (S3) and (S5) of Bailey pairs are inherited from \cite{BIS}.

If we take $(\rho_1,\rho_2)\rightarrow (\infty,-q^{1/2})$, we obtain the Bailey pair \cite[Eq.\ (S3)]{BIS}:
\begin{align}\label{BL3}
\begin{split}
   & \alpha_n'(a;q)=\frac{(-q^{1/2};q)_n}{(-aq^{1/2};q)_n}a^nq^{n^2/2}\alpha_n(a;q), \\
   & \beta_n'(a;q)=\sum_{r=0}^n\frac{(-q^{1/2};q)_r}{(q;q)_{n-r}(-aq^{1/2};q)_n}a^rq^{r^2/2}\beta_r(a;q).
\end{split} \tag{S3}
\end{align}

If we take $(\rho_1,\rho_2)\rightarrow (\infty,-a^{1/2}q)$, we obtain the Bailey pair \cite[Eq.\ (S5)]{BIS}:
\begin{align}\label{BL5}
\begin{split}
    &\alpha_n'(a;q)=\frac{(-a^{1/2}q;q)_n}{(-a^{1/2};q)_n}a^{n/2}q^{(n^2-n)/2}\alpha_n(a;q), \\
    &\beta_n'(a;q)=\sum_{r=0}^n \frac{(-a^{1/2}q;q)_r}{(q;q)_{n-r}(-a^{1/2};q)_n}a^{r/2}q^{(r^2-r)/2}\beta_r(a;q).
\end{split} \tag{S5}
\end{align}

If we let $n,\rho_1,\rho_2\to \infty$ in \eqref{eq-Bailey-general-id}, we obtain the following result (see, e.g. \cite[Eq.\ (1.2.8)]{MSZ}).
\begin{lemma}
If $(\alpha_n(a;q),\beta_n(a;q))$ is a Bailey pair, we have
\begin{align}
    \sum_{n=0}^\infty a^nq^{n^2}\beta_n(a;q)=\frac{1}{(aq;q)_\infty}\sum_{n=0}^{\infty}a^nq^{n^2}\alpha_n(a;q). \label{cor1}
\end{align}
\end{lemma}

We need some known Bailey pairs from \cite[p.\ 469]{Slater1951}. We write $\alpha_n(a,q)$ in a unified expression while Slater \cite{Slater1951} wrote different expressions for even  and odd values of $n$ separately. Below we always let $\alpha_{0}=1$. The expressions of Bailey pairs for $n\geq 1$ are given as follows:
\begin{align}
    &\alpha_n(1;q)=(-1)^nq^{\frac{1}{2}n^2+\frac{1}{2}\binom{n}{2}}(1+q^{\frac{1}{2}n}), \quad \beta_n(1;q)=\frac{1}{(q^2;q^2)_n(-q^{\frac{1}{2}};q)_n}; \tag{G1}
    \label{B-G(1)}
    \\
    &\alpha_n(q;q)=(-1)^n q^{\frac{3}{2}\binom{n+1}{2}}\frac{q^{-{n}/{2}}-q^{{(n+1)}/{2}}}{1-q^{{1}/{2}}}, \quad \beta_n(q;q)=\frac{1}{(q^2;q^2)_n(-q^{{3}/{2}};q)_n};   \tag{G2}\label{G(2)} \\
    &\alpha_n(1;q)=(-1)^nq^{\frac{3}{2}\binom{n}{2}}(1+q^{\frac{3}{2}n}), \quad \beta_n(1;q)=\frac{q^n}{(q^2;q^2)_n(-q^{1/2};q)_n}. \tag{G3} \label{G(3)}
\end{align}

For our purpose, we also need the following Bailey pair which appears to be new:
\begin{equation}
\begin{split}
 &\alpha_n(q;q)=(-1)^nq^{\frac{3}{2}\binom{n+1}{2}}\frac{q^{-n}-q^{n+1}}{1-q}, \\
 &\beta_n(q;q)=\frac{1}{(q^2;q^2)_n(-q^{{1}/{2}};q)_n}.
\end{split} \tag{G1*} \label{B-G(1.1)}
\end{equation}
To verify that it is indeed a Bailey pair, we need the identity
\cite[Eq.\ (4.2)]{Slater1951}:
\begin{align}\label{bilateral-id}
    &\sum_{r=[-n/2]}^{[n/2]}\frac{(1-aq^{4r})(q^{-n};q)_{2r}a^{2r}q^{2nr+r}(d;q^2)_r(e;q^2)_r}{(1-a)(aq^{n+1};q)_{2r}d^re^r(aq^2/d;q^2)_r(aq^2/e;q^2)_r}\nonumber \\
    &=\frac{(q^2/a,aq/d,aq/e,aq^2/de;q^2)_\infty}{(q,q^2/d,q^2/e,a^2q/de;q^2)_\infty}\frac{(q;q)_n(aq;q)_n(a^2q/de;q^2)_n}{(aq;q^2)_n(aq/d,aq/e;q)_n}
\end{align}
where  $a$ has such a value that the series terminates below at $r=[-n/2]$.
\begin{lemma}\label{lem-new-Bailey}
We have
\begin{align}
    \sum_{i=0}^k\frac{(-1)^iq^{\frac{3}{4}i^2-\frac{1}{4}i}(1-q^{2i+1})}{(1-q)(q^2;q)_{k+i}(q;q)_{k-i}}=\frac{1}{(-q^{{1}/{2}};q)_k(q^2;q^2)_k}.\label{G(1)-new}
\end{align}
\end{lemma}
This lemma proves that \eqref{B-G(1.1)} is indeed a Bailey pair.
\begin{proof}
By definition \eqref{aq-defn}, for $r,k\geq 0$ we have
\begin{align}
    (q^{-k};q)_{2r}=(-1)^kq^{2r^2-2kr-r}(1-q^k)(1-q^{k-1})\cdots (1-q^{k-2r+1}), \label{prod-1}\\
    \frac{(q^{-k};q)_{-2r-2}}{(q^{k+2};q)_{-2r-2}}=q^{(r+1)(2k+2r+3)}\frac{(1-q^{k-2r})(1-q^{k-2r+1})\cdots (1-q^{k+1})}{(1-q^{k+1})(1-q^{k+2})\cdots (1-q^{k+2r+2})}. \label{prod-2}
\end{align}
Setting $a=q, d=-q^{\frac{3}{2}}, e\to \infty$ in \eqref{bilateral-id}, we have
    \begin{align*}
        &\frac{(q;q)_k(q^2;q)_k}{(q^2;q^2)_k(-q^{\frac{1}{2}};q)_k}=
        \sum_{r=[-k/2]}^{[k/2]}\frac{(1-q^{4r+1})(q^{-k};q)_{2r}q^{2kr+r^2+\frac{1}{2}r}}{(1-q)(q^{k+2};q)_{2r}}
         \\
        &=\sum_{r\ge 0}\frac{(1-q^{4r+1})(q^{-k};q)_{2r}q^{2kr+r^2+\frac{1}{2}r}}{(1-q)(q^{k+2};q)_{2r}}+\sum_{r < 0}\frac{(1-q^{4r+1})(q^{-k};q)_{2r}q^{2kr+r^2+\frac{1}{2}r}}{(1-q)(q^{k+2};q)_{2r}}
        \\
         &=\sum_{r\ge 0}\frac{(1-q^{4r+1})(q^{-k};q)_{2r}q^{2kr+r^2+\frac{1}{2}r}}{(1-q)(q^{k+2};q)_{2r}}  \\
        &\quad \quad \quad
        +\sum_{r\geq 0}\frac{(1-q^{-4r-3})(q^{-k};q)_{-2r-2}q^{-2k(r+1)+(r+1)^2-\frac{1}{2}(r+1)}}{(1-q)(q^{k+2};q)_{-2r-2}}
        \\
        &=(q;q)_k(q^2;q)_k\Big(\sum_{\begin{smallmatrix}
i=2r\geq 0
\end{smallmatrix}}\frac{(-1)^iq^{\frac{3}{4}i^2-\frac{1}{4}i}(1-q^{2i+1})}{(1-q)(q^2;q)_{k+i}(q;q)_{k-i}} \nonumber \\
&\quad \quad  \quad \quad    \quad \quad   \quad \quad   \quad \quad        +
        \sum_{\begin{smallmatrix}
i=2r+1\geq 0
\end{smallmatrix}}\frac{(-1)^iq^{\frac{3}{4}i^2-\frac{1}{4}i}(1-q^{2i+1})}{(1-q)(q^2;q)_{k+i}(q;q)_{k-i}}\Big)
\\
    &=(q;q)_k(q^2;q)_k\sum_{i=0}^k\frac{(-1)^iq^{\frac{3}{4}i^2-\frac{1}{4}i}(1-q^{2i+1})}{(1-q)(q^2;q)_{k+i}(q;q)_{k-i}}.
    \end{align*}
Here for the third equality we replaced $r$ by $-r-1$ in the second sum, and for the fourth equality we used \eqref{prod-1} and \eqref{prod-2} for simplifications.   This proves \eqref{G(1)-new}.
\end{proof}

\begin{lemma}\label{lem-DJK}
(Cf. \cite[Lemma 1.12]{DJK}).
If $(\alpha_n(a;q),\beta_n(a;q))$ is a Bailey pair related to $a$, then
\begin{equation}\label{Douse-lemma1.12}
\begin{split}
    &\alpha_n'(a/q;q):=(1-a)\Big(\frac{1-bq^n}{1-b}\frac{\alpha_n}{1-aq^{2n}}-\frac{q^{n-1}(aq^{n-1}-b)}{1-b}\frac{\alpha_{n-1}}{1-aq^{2n-2}}\Big) \\
    &\beta_n'(a/q;q):=\frac{(bq;q)_n}{(b;q)_n}\beta_n
\end{split}
\end{equation}
is a Bailey pair related to $a/q$.
\end{lemma}

\begin{lemma}\label{lem-BP-mod}
If $(\alpha_n(q;q),\beta_n(q;q))$ is a Bailey pair related to $q$ with
\begin{align}
\alpha_n(q;q)=\frac{(-1)^nu^{\binom{n+1}{2}}(q^{-n}-q^{n+1})}{1-q},
\end{align}
then $(\alpha'_n(1;q),\beta'_n(1;q))$ is a Bailey pair related to $1$ where
\begin{align}
\alpha_0'=1,\quad \alpha'_{n}(1;q)=(-1)^nu^{\binom{n}{2}}(1+u^n), \quad  \beta_n'(1;q)=q^n\beta_n(1;q).
\end{align}
\end{lemma}
\begin{proof}
    Setting $b \to \infty$ in \eqref{Douse-lemma1.12}, we obtain the Bailey pair $(\alpha_n',\beta_n')$.
\end{proof}

We will not use Lemmas \ref{lem-DJK} and \ref{lem-BP-mod} until Section \ref{sec-general}, but here we point out an interesting fact. If we set  $u=q^{\frac{3}{2}}$ in Lemma \ref{lem-BP-mod}, then we obtain the \eqref{G(3)} Bailey pair immediately from the \eqref{B-G(1.1)} Bailey pair.

\section{Identities for Mizuno's rank three examples in Table 2}\label{sec-exam}
Since $c$ is uniquely determined by $A,b,d$ when $(A,b,c,d)$ is a modular quadruple, we will not mention it. In this section we set $d=(1,1,2)$ and $D=\diag\{1,1,2\}$. We will divide Mizuno's examples into seven groups. The rank three Nahm sums for examples in the same group can either be reduced to same set of lower rank Nahm sums or have similar product representations. For each example, we will present a set of Rogers--Ramanujan type identities which give modular product representations for the Nahm sums involved.

\subsection{Examples 1, 6, 7 and 10}

We will reduce the Nahm sums in these examples to double sums and use the following identities:
\begin{align}
 &\sum_{i,j \ge 0}\frac{q^{2i^2+ij+\frac{1}{2}j^2+\frac{1}{2}j}}{(q;q)_i(q;q)_j}
    =
    \frac{(q^4,q^6,q^{10};q^{10})_\infty}{(q;q)_\infty}, \quad \text{(\cite[Eq.\ (3.28)]{Wang-rank2})} \label{LW.rank2.3.28}
    \\
    &\sum_{i,j \ge 0}\frac{q^{2i^2+ij+\frac{1}{2}j^2+2i+\frac{1}{2}j}}{(q;q)_i(q;q)_j}
    =
    \frac{(q^2,q^8,q^{10};q^{10})_\infty}{(q;q)_\infty}. \quad \text{(\cite[Eq.\ (3.29)]{Wang-rank2})} \label{LW.rank2.3.29}
 \end{align}
\subsubsection{Example 1}
This example corresponds to
\begin{align*}
A=\begin{pmatrix}
1 & 1 & 0 \\
1 & 2 & 1 \\
0 & 2 & 4
\end{pmatrix}, \quad AD=\begin{pmatrix}
1 & 1 & 0 \\ 1 & 2 & 2 \\ 0 & 2 & 8
\end{pmatrix}, \quad
b \in \bigg\{
\begin{pmatrix}
1/2 \\ 1 \\ 0
\end{pmatrix},
\begin{pmatrix}
1/2 \\ 1 \\ 4
\end{pmatrix}
\bigg\}.
\end{align*}
\begin{theorem}\label{thm-1}
We have
\begin{align}
\sum_{i,j,k\ge 0}\frac{q^{\frac{1}{2}i^2+j^2+4k^2+ij+2jk+\frac{1}{2}i+j}}{(q;q)_i(q;q)_j(q^2;q^2)_k}
=
\frac{(q^8,q^{12},q^{20};q^{20})_\infty}{(q;q)_\infty},
\label{table2.1.1}
\\
\sum_{i,j,k\ge 0}\frac{q^{\frac{1}{2}i^2+j^2+4k^2+ij+2jk+\frac{1}{2}i+j+4k}}{(q;q)_i(q;q)_j(q^2;q^2)_k}
=\frac{(q^4,q^{16},q^{20};q^{20})_\infty}{(q;q)_\infty}.
\label{table2.1.2}
\end{align}
\end{theorem}
\begin{proof}
We have
\begin{align}
&\sum_{i,j,k\ge 0}\frac{q^{\frac{1}{2}i^2+j^2+4k^2+ij+2jk}u^iv^jw^k}{(q;q)_i(q;q)_j(q^2;q^2)_k}\notag
\\
&=
\sum_{j,k\ge 0}\frac{q^{j^2+2jk+4k^2}v^jw^k}{(q;q)_j(q^2;q^2)_k}
\sum_{i\ge 0}\frac{q^{\frac{1}{2}(i^2-i)}(uq^{\frac{1}{2}+j})^i}{(q;q)_i}
\notag
\\
&=\sum_{j,k\ge 0}\frac{q^{j^2+2jk+4k^2}v^jw^k(-uq^{\frac{1}{2}+j};q)_\infty}{(q;q)_j(q^2;q^2)_k}  \quad \text{(by \eqref{euler})}\notag
\\
&=(-uq^\frac{1}{2};q)_\infty\sum_{j,k\ge 0}\frac{q^{j^2+2jk+4k^2}v^jw^k}{(q;q)_j(q^2;q^2)_k(-uq^\frac{1}{2};q)_j}\label{F2.1}.
\end{align}
Setting $(u,v,w)=(q^\frac{1}{2},q,1)$ and $(q^{\frac{1}{2}},q,q^4)$ in \eqref{F2.1}, by \eqref{LW.rank2.3.28} and \eqref{LW.rank2.3.29} we obtain \eqref{table2.1.1} and  \eqref{table2.1.2}, respectively.
\end{proof}

\subsubsection{Example 6}
This example corresponds to
\begin{align*}
A=\begin{pmatrix}
1 & 1 & 1\\
1 & 5 & 4\\
2 & 8 & 8
\end{pmatrix},\quad
AD=\begin{pmatrix}
    1 & 1 & 2\\
    1 & 5 & 8\\
    2 & 8 & 16
\end{pmatrix}
,\quad
b\in  \bigg\{
\begin{pmatrix}
1/2 \\ -1/2 \\ 0
\end{pmatrix},
\begin{pmatrix}
1/2 \\ 3/2 \\ 4
\end{pmatrix}
\bigg\}.
\end{align*}
\begin{theorem}\label{thm-6}
We have
\begin{align}
\sum_{i,j,k\ge 0}\frac{q^{\frac{1}{2}i^2+\frac{5}{2}j^2+8k^2+ij+2ik+8jk+\frac{1}{2}i-\frac{1}{2}j}}{(q;q)_i(q;q)_j(q^2;q^2)_k}
&=\frac{(q^4,q^6,q^{10};q^{10})_\infty}{(q;q)_\infty}, \label{table2.6.1}
\\
\sum_{i,j,k\ge 0}\frac{q^{\frac{1}{2}i^2+\frac{5}{2}j^2+8k^2+ij+2ik+8jk+\frac{1}{2}i+\frac{3}{2}j+4k}}{(q;q)_i(q;q)_j(q^2;q^2)_k}
&=\frac{(q^2,q^8,q^{10};q^{10})_\infty}{(q;q)_\infty}. \label{table2.6.2}
\end{align}
\end{theorem}
\begin{proof}
We have
\begin{align}
&\sum_{i,j,k\ge 0}\frac{q^{\frac{1}{2}i^2+\frac{5}{2}j^2+8k^2+ij+2ik+8jk-\frac{1}{2}j}u^iv^{j+2k}}{(q;q)_i(q;q)_j(q^2;q^2)_k}
\notag
\\
&=\sum_{i\ge 0}\frac{q^{\frac{1}{2}i^2}u^i}{(q;q)_i}\sum_{m\ge 0}q^{2m^2+im}v^m\sum_{j+2k=m}\frac{q^{\frac{1}{2}(j^2-j)}}{(q;q)_j(q^2;q^2)_k}
\notag
\\
&=\sum_{i,m\ge 0}\frac{q^{\frac{1}{2}i^2+im+2m^2}u^iv^m}{(q;q)_i(q;q)_m}. \label{F2.6}
\end{align}
Here for the last equality we used Lemma \ref{lem-12}.

Setting $(u,v) = (q^{\frac{1}{2}},1)$ and $(q^{\frac{1}{2}},q^2)$, by \eqref{LW.rank2.3.28} and \eqref{LW.rank2.3.29} we obtain \eqref{table2.6.1} and \eqref{table2.6.2}, respectively.
\end{proof}

Since the proofs of some examples below are similar to the previous examples, we will omit some details.
\subsubsection{Example 7}
This example corresponds to
\begin{align*}
A=\begin{pmatrix}
1 & 1 & 0\\
1 & 8 & 1\\
0 & 2 & 1
\end{pmatrix},\quad
AD=\begin{pmatrix}
    1 & 1 & 0\\
    1 & 8 & 2\\
    0 & 2 & 2
\end{pmatrix}
,\quad
b\in  \bigg\{
\begin{pmatrix}
1/2 \\ 0 \\ 1
\end{pmatrix},
\begin{pmatrix}
1/2 \\ 4 \\ 1
\end{pmatrix}
\bigg\}.
\end{align*}
\begin{theorem}\label{thm-7}
We have
\begin{align}
\sum_{i,j,k\ge 0}\frac{q^{\frac{1}{2}i^2+4j^2+k^2+ij+2jk+\frac{1}{2}i+k}}{(q;q)_i(q;q)_j(q^2;q^2)_k}
&=\frac{(q^8,q^{12},q^{20};q^{20})_\infty}{(q;q)_\infty}, \label{table2.7.1}
\\
\sum_{i,j,k\ge 0}\frac{q^{\frac{1}{2}i^2+4j^2+k^2+ij+2jk+\frac{1}{2}i+4j+k}}{(q;q)_i(q;q)_j(q^2;q^2)_k}
&=\frac{(q^4,q^{16},q^{20};q^{20})_\infty}{(q;q)_\infty}. \label{table2.7.2}
\end{align}
\end{theorem}
\begin{proof}
Summing over $i$ using \eqref{euler} first, we have
\begin{align}
\sum_{i,j,k\ge 0}\frac{q^{\frac{1}{2}i^2+4j^2+k^2+ij+2jk}u^iv^jw^k}{(q;q)_i(q;q)_j(q^2;q^2)_k}
=(-uq^{\frac{1}{2}};q)_\infty\sum_{j,k\ge 0}\frac{q^{4j^2+2jk+k^2}v^jw^k}{(q;q)_j(q^2;q^2)_k(-uq^{\frac{1}{2}};q)_j}. \label{F2.7}
\end{align}
Setting $(u,v,w)=(q^\frac{1}{2},1,q)$ and $(q^\frac{1}{2},q^4,q)$, and then using \eqref{LW.rank2.3.28}  and \eqref{LW.rank2.3.29}, we obtain \eqref{table2.7.1} and \eqref{table2.7.2}, respectively.
\end{proof}

\subsubsection{Example 10}
This example corresponds to
\begin{align*}
A=\begin{pmatrix}
2 & 1 & 1 \\
1 & 4 & 1 \\
2 & 2 & 2
\end{pmatrix},
\quad
AD=\begin{pmatrix}
    2 & 1 & 2\\
    1 & 4 & 2\\
    2 & 2 & 4
\end{pmatrix}
,\quad
b\in  \bigg\{
\begin{pmatrix}
0 \\ 0 \\ 1
\end{pmatrix},
\begin{pmatrix}
0 \\ 2 \\ 1
\end{pmatrix}
\bigg\}.
\end{align*}
\begin{theorem}\label{thm-10}
We have
\begin{align}
\sum_{i,j,k\ge 0}\frac{q^{i^2+2j^2+2k^2+ij+2ik+2jk+k}}{(q;q)_i(q;q)_j(q^2;q^2)_k}
&=\frac{(q^4,q^6,q^{10};q^{10})_\infty}{(q;q)_\infty}, \label{table2.10.1}
\\
\sum_{i,j,k\ge 0}\frac{q^{i^2+2j^2+2k^2+ij+2ik+2jk+2j+k}}{(q;q)_i(q;q)_j(q^2;q^2)_k}
&=
\frac{(q^2,q^8,q^{10};q^{10})_\infty}{(q;q)_\infty}.
\label{table2.10.2}
\end{align}
\end{theorem}

\begin{proof}
Summing over $i,k$ with $i+2k=m$ using Lemma \ref{lem-12} first, we have
\begin{align}
\sum_{i,j,k\ge 0}\frac{q^{i^2+2j^2+2k^2+ij+2ik+2jk-\frac{1}{2}i}u^{i+2k}v^j}{(q;q)_i(q;q)_j(q^2;q^2)_k}
=\sum_{j,m\ge 0}\frac{q^{2j^2+jm+\frac{1}{2}m^2}u^mv^j}{(q;q)_j(q;q)_m}. \label{F2.10}
\end{align}
Setting $(u,v)=(q^{\frac{1}{2}},1)$ and $(q^{\frac{1}{2}},q^2)$, and then using \eqref{LW.rank2.3.28} and \eqref{LW.rank2.3.29}, we obtain \eqref{table2.10.1} and \eqref{table2.10.2}, respectively.
\end{proof}

\subsection{Examples 2, 16, 18 and 19}
We will reduce the Nahm sums in these examples to the following double Nahm sums:
\begin{align}
    &\sum_{i,j\ge 0}\frac{q^{3i^2+4ij+4j^2-2i}}{(q^4;q^4)_i(q^4;q^4)_j}
    =
    \frac{J_{14}J_{28}^2J_{2,28}}{J_{1,28}J_{4,28}J_{8,28}J_{13,28}}, \quad \text{(\cite[Eq.\ (3.70)]{Wang-rank2})}
    \label{LW.rank2.3.70}
    \\
    &\sum_{i,j\ge 0}\frac{q^{3i^2+4ij+4j^2}}{(q^4;q^4)_i(q^4;q^4)_j}
    =
    \frac{J_{14}J_{28}^2J_{6,28}}{J_{3,28}J_{4,28}J_{11,28}J_{12,28}}, \quad \text{(\cite[Eq.\ (3.71)]{Wang-rank2})}
    \label{LW.rank2.3.71}
    \\
    &\sum_{i,j\ge 0}\frac{q^{3i^2+4ij+4j^2+2i+4j}}{(q^4;q^4)_i(q^4;q^4)_j}
    =
    \frac{J_{14}J_{28}^2J_{10,28}}{J_{5,28}J_{8,28}J_{9,28}J_{12,28}}. \quad \text{(\cite[Eq.\ (3.72)]{Wang-rank2})}
    \label{LW.rank2.3.72}
\end{align}

\subsubsection{Example 2}
This example corresponds to
\begin{align*}
A=\begin{pmatrix}
1 & 1 & 0 \\ 1 & 3 & 1 \\ 0 & 2 & 2
\end{pmatrix},\quad
AD=\begin{pmatrix}
    1 & 1 & 0\\
    1 & 3 & 2\\
    0 & 2 & 4
\end{pmatrix}
,\quad
b \in \bigg\{
\begin{pmatrix}
1/2 \\ -1 \\ 0
\end{pmatrix},
\begin{pmatrix}
1/2 \\ 0 \\ 0
\end{pmatrix},
\begin{pmatrix}
1/2 \\ 1 \\ 2
\end{pmatrix}
\bigg\}.
\end{align*}
\begin{theorem}\label{thm-2}
We have
\begin{align}
\sum_{i,j,k\ge 0}\frac{q^{i^2+3j^2+4k^2+2ij+4jk+i-2j}}{(q^2;q^2)_i(q^2;q^2)_j(q^4;q^4)_k}
=\frac{J_4J_{14}^2J_{28}J_{2,28}}{J_2J_{1,28}J_{4,28}J_{8,28}J_{13,28}},
\label{table2.2.1}
\\
\sum_{i,j,k\ge 0}\frac{q^{i^2+3j^2+4k^2+2ij+4jk+i}}{(q^2;q^2)_i(q^2;q^2)_j(q^4;q^4)_k}
=\frac{J_4J_{14}^2J_{28}J_{6,28}}{J_2J_{3,28}J_{4,28}J_{11,28}J_{12,28}},
\label{table2.2.2}
\\
\sum_{i,j,k\ge 0}\frac{q^{i^2+3j^2+4k^2+2ij+4jk+i+2j+4k}}{(q^2;q^2)_i(q^2;q^2)_j(q^4;q^4)_k}
=
\frac{J_4J_{14}^2J_{28}J_{10,28}}{J_2J_{5,28}J_{8,28}J_{9,28}J_{12,28}}.
\label{table2.2.3}
\end{align}
\end{theorem}
\begin{proof}
Summing over $i$ using \eqref{euler} first, we have
\begin{align}
&\sum_{i,j,k\ge 0}\frac{q^{i^2+3j^2+4k^2+2ij+4jk}u^iv^jw^k}{(q^2;q^2)_i(q^2;q^2)_j(q^4;q^4)_k}
\notag \\
&=(-qu;q^2)_\infty\sum_{j,k\ge 0}\frac{q^{3j^2+4jk+4k^2}v^jw^k}{(q^2;q^2)_j(q^4;q^4)_k(-qu;q^2)_j}. \label{F2.2}
\end{align}
Setting $(u,v,w)=(q,q^{-2},1)$, $(q,1,1)$ and $(q,q^2,q^4)$ in \eqref{F2.2}, by \eqref{LW.rank2.3.70}--\eqref{LW.rank2.3.72} we obtain \eqref{table2.2.1}--\eqref{table2.2.3}, respectively.
\end{proof}

\subsubsection{Example 16}
This example corresponds to
\begin{align*}
A=\begin{pmatrix}
1 & 1 & 0\\
1 & 4 & 1\\
0 & 2 & 3/2
\end{pmatrix},\quad
AD=\begin{pmatrix}
    1 & 1 & 0\\
    1 & 4 & 2\\
    0 & 2 & 3
\end{pmatrix}
,\quad
b\in  \bigg\{
\begin{pmatrix}
1/2\\ 0 \\-1
\end{pmatrix},
\begin{pmatrix}
1/2\\ 0 \\ 0
\end{pmatrix},
\begin{pmatrix}
1/2\\ 2 \\ 1
\end{pmatrix}
\bigg\}.
\end{align*}
\begin{theorem}
We have
\begin{align}
\sum_{i,j,k\ge 0}\frac{q^{i^2+4j^2+3k^2+2ij+4jk+i-2k}}{(q^2;q^2)_i(q^2;q^2)_j(q^4;q^4)_k}
&=\frac{J_4J_{14}J_{28}^2J_{2,28}}{J_2J_{1,28}J_{4,28}J_{8,28}J_{13,28}},
\label{table2.16.1}
\\
\sum_{i,j,k\ge 0}\frac{q^{i^2+4j^2+3k^2+2ij+4jk+i}}{(q^2;q^2)_i(q^2;q^2)_j(q^4;q^4)_k}
&=\frac{J_4J_{14}J_{28}^2J_{6,28}}{J_2J_{3,28}J_{4,28}J_{11,28}J_{12,28}},
\label{table2.16.2}
\\
\sum_{i,j,k\ge 0}\frac{q^{i^2+4j^2+3k^2+2ij+4jk+i+4j+2k}}{(q^2;q^2)_i(q^2;q^2)_j(q^4;q^4)_k}
&=\frac{J_4J_{14}J_{28}^2J_{10,28}}{J_2J_{5,28}J_{8,28}J_{9,28}J_{12,28}}.
\label{table2.16.3}
\end{align}
\end{theorem}
\begin{proof}
Summing over $i$ first using \eqref{euler}, we have
\begin{align}
&\sum_{i,j,k\ge 0}\frac{q^{i^2+4j^2+3k^2+2ij+4jk}u^iv^jw^k}{(q^2;q^2)_i(q^2;q^2)_j(q^4;q^4)_k}  \nonumber \\
&=(-uq;q^2)_\infty\sum_{j,k\ge 0}\frac{q^{4j^2+4jk+3k^2}v^jw^k}{(q^2;q^2)_j(q^4;q^4)_k(-uq;q^2)_j}.
\label{F2.16}
\end{align}
Setting $(u,v,w)=(q,1,q^{-2})$, $(q,1,1)$ and $(q,q^4,q^2)$ in \eqref{F2.16}, by \eqref{LW.rank2.3.70}--\eqref{LW.rank2.3.72}, we obtain \eqref{table2.16.1}--\eqref{table2.16.3}, respectively.
\end{proof}

\subsubsection{Example 18}
This example corresponds to
\begin{align*}
A=\begin{pmatrix}
3/2&1 & 1\\
1 & 3 & 2\\
2 & 4 & 4
\end{pmatrix},\quad
AD=\begin{pmatrix}
    3/2&1 & 2\\
    1 & 3 & 4\\
    2 & 4 & 8
\end{pmatrix}
,\quad
b\in  \bigg\{
\begin{pmatrix}
-1/2\\-1/2\\ 0
\end{pmatrix},
\begin{pmatrix}
0 \\-1/2\\ 0
\end{pmatrix},
\begin{pmatrix}
1/2\\1/2\\ 2
\end{pmatrix}
\bigg\}.
\end{align*}
\begin{theorem}
We have
\begin{align}
\sum_{i,j,k\ge 0}\frac{q^{3i^2+6j^2+16k^2+4ij+8ik+16jk-2i-2j}}{(q^4;q^4)_i(q^4;q^4)_j(q^8;q^8)_k}
&=\frac{J_{14}J_{28}^2J_{2,28}}{J_{1,28}J_{4,28}J_{8,28}J_{13,28}},
\label{table2.18.1}
\\
\sum_{i,j,k\ge 0}\frac{q^{3i^2+6j^2+16k^2+4ij+8ik+16jk-2j}}{(q^4;q^4)_i(q^4;q^4)_j(q^8;q^8)_k}
&=\frac{J_{14}J_{28}^2J_{6,28}}{J_{3,28}J_{4,28}J_{11,28}J_{12,28}},
\label{table2.18.2}
\\
\sum_{i,j,k\ge 0}\frac{q^{3i^2+6j^2+16k^2+4ij+8ik+16jk+2i+2j+8k}}{(q^4;q^4)_i(q^4;q^4)_j(q^8;q^8)_k}
&=\frac{J_{14}J_{28}^2J_{10,28}}{J_{5,28}J_{8,28}J_{9,28}J_{12,28}}.
\label{table2.18.3}
\end{align}
\end{theorem}
\begin{proof}
Summing over $j,k$ with $j+2k=m$ using Lemma \ref{lem-12} first, we have
\begin{align}
\sum_{i,j,k\ge 0}\frac{q^{3i^2+6j^2+16k^2+4ij+8ik+16jk-2j}u^iv^{j+2k}}{(q^4;q^4)_i(q^4;q^4)_j(q^8;q^8)_k}
=\sum_{i,m\ge 0}\frac{q^{3i^2+4mi+4m^2}u^iv^m}{(q^4;q^4)_i(q^4;q^4)_m}.
\label{F2.18}
\end{align}
Setting $(u,v)=(q^{-2},1)$, $(1,1)$ and $(q^2,q^4)$ in \eqref{F2.18},  and then using  \eqref{LW.rank2.3.70}--\eqref{LW.rank2.3.72}, we obtain \eqref{table2.18.1}--\eqref{table2.18.3}, respectively.
\end{proof}

\subsubsection{Example 19}
This example corresponds to
\begin{align*}
A=\begin{pmatrix}
2 & 1 & 1\\
1 & 5/2 & 3/2\\
2 & 3 & 3
\end{pmatrix},\quad
AD=\begin{pmatrix}
    2 & 1 & 2\\
    1 &5/2& 3\\
    2 & 3 & 6
\end{pmatrix}
,\quad
b\in  \bigg\{
\begin{pmatrix}
0 \\ -1 \\ -1
\end{pmatrix},
\begin{pmatrix}
0 \\-1/2 \\ 0
\end{pmatrix},
\begin{pmatrix}
1 \\ 0 \\ 1
\end{pmatrix}
\bigg\}.
\end{align*}
\begin{theorem}
We have
\begin{align}
\sum_{i,j,k\ge 0}\frac{q^{4i^2+5j^2+12k^2+4ij+8ik+12jk-4j-4k}}{(q^4;q^4)_i(q^4;q^4)_j(q^8;q^8)_k}
&=\frac{J_{14}J_{28}^2J_{2,28}}{J_{1,28}J_{4,28}J_{8,28}J_{13,28}},
\label{table2.19.1}
\\
\sum_{i,j,k\ge 0}\frac{q^{4i^2+5j^2+12k^2+4ij+8ik+12jk-2j}}{(q^4;q^4)_i(q^4;q^4)_j(q^8;q^8)_k}
&=\frac{J_{14}J_{28}^2J_{6,28}}{J_{3,28}J_{4,28}J_{11,28}J_{12,28}},
\label{table2.19.2}
\\
\sum_{i,j,k\ge 0}\frac{q^{4i^2+5j^2+12k^2+4ij+8ik+12jk+4i+4k}}{(q^4;q^4)_i(q^4;q^4)_j(q^8;q^8)_k}
&=\frac{J_{14}J_{28}^2J_{10,28}}{J_{5,28}J_{8,28}J_{9,28}J_{12,28}}.
\label{table2.19.3}
\end{align}
\end{theorem}
\begin{proof}
Summing over $j,k$ with $j+2k=m$ using Lemma \ref{lem-12} first, we have
\begin{align}
\sum_{i,j,k\ge 0}\frac{q^{4i^2+5j^2+12k^2+4ij+8ik+12jk-2j}u^iv^{j+2k}}{(q^4;q^4)_i(q^4;q^4)_j(q^8;q^8)_k}
=\sum_{i,m\ge 0}\frac{q^{4i^2+4im+3m^2}u^iv^m}{(q^4;q^4)_i(q^4;q^4)_m}.
\label{F2.19}
\end{align}
Setting $(u,v)=(1,q^{-2})$, $(1,1)$ and $(q^4,q^2)$ in \eqref{F2.19}, and then using \eqref{LW.rank2.3.70}--\eqref{LW.rank2.3.72}, we obtain \eqref{table2.19.1}--\eqref{table2.19.3}, respectively.
\end{proof}

\subsection{Examples 3, 4 and 9}
We will reduce the Nahm sums in these examples to the following double Nahm sums:
\begin{align}
    &\sum_{i,j\ge 0}\frac{q^{i^2+ij+\frac{1}{2}j^2-i+\frac{1}{2}j}}{(q;q)_i(q;q)_j}
    =
    2\frac{(q^4;q^4)_\infty}{(q;q)_\infty}, \quad \text{(\cite[Eq.\ (3.5)]{Wang-rank2})}\label{LW.rank2.3.5}
    \\
    &\sum_{i,j\ge 0}\frac{q^{2i^2+2ij+j^2}}{(q^2;q^2)_i(q^2;q^2)_j}
    =
    \frac{1}{(q,q^4,q^7;q^8)_\infty}, \quad \text{(\cite[Eq.\ (3.6)]{Wang-rank2})} \label{LW.rank2.3.6}
    \\
    &\sum_{i,j\ge 0}\frac{q^{i^2+ij+\frac{1}{2}j^2+\frac{1}{2}j}}{(q;q)_i(q;q)_j}
    =
    \frac{(q^2;q^2)^3_\infty}{(q;q)^2_\infty(q^4;q^4)_\infty}, \quad \text{(\cite[Eq.\ (3.7)]{Wang-rank2})} \label{LW.rank2.3.7}
    \\
    &\sum_{i,j\ge 0}\frac{q^{i^2+ij+\frac{1}{2}j^2+i+\frac{1}{2}j}}{(q;q)_i(q;q)_j}
    =
    \frac{(q^4;q^4)_\infty}{(q;q)_\infty},  \quad \text{(\cite[Eq.\ (3.8)]{Wang-rank2})} \label{LW.rank2.3.8}
    \\
    &\sum_{i,j\ge 0}\frac{q^{2i^2+2ij+j^2+2i+2j}}{(q^2;q^2)_i(q^2;q^2)_j}
    =
    \frac{1}{(q^3,q^4,q^5;q^8)_\infty}. \quad \text{(\cite[Eq.\ (3.9)]{Wang-rank2})} \label{LW.rank2.3.9}
\end{align}

\subsubsection{Example 3}
This example corresponds to
\begin{align*}
&A=\begin{pmatrix}
1 & 1 & 1 \\
1 & 3 & 2 \\
2 & 4 & 4
\end{pmatrix},\quad
AD=\begin{pmatrix}
    1 & 1 & 2\\
    1 & 3 & 4\\
    2 & 4 & 8
\end{pmatrix},\nonumber
\\
&b\in  \bigg\{
\begin{pmatrix}
0 \\ -1/2 \\ 0
\end{pmatrix},
\begin{pmatrix}
1/2 \\ -3/2 \\ -2
\end{pmatrix},
\begin{pmatrix}
1/2 \\ -1/2 \\ 0
\end{pmatrix},
\begin{pmatrix}
1/2 \\ 1/2 \\ 2
\end{pmatrix},
\begin{pmatrix}
1 \\ 1/2 \\ 2
\end{pmatrix}
\bigg\}.
\end{align*}
\begin{theorem}\label{thm-3}
We have
\begin{align}
\sum_{i,j,k\ge 0}\frac{q^{i^2+3j^2+8k^2+2ij+4ik+8jk-j}}{(q^2;q^2)_i(q^2;q^2)_j(q^4;q^4)_k}
&=
\frac{1}{(q,q^4,q^7;q^8)_\infty}, \label{table2.3.1}
\\
\sum_{i,j,k\ge 0}\frac{q^{i^2+3j^2+8k^2+2ij+4ik+8jk+i-3j-4k}}{(q^2;q^2)_i(q^2;q^2)_j(q^4;q^4)_k}
&=2\frac{(q^2;q^2)_\infty}{(q^8;q^8)_\infty}, \label{table2.3.2}
\\
\sum_{i,j,k\ge 0}\frac{q^{i^2+3j^2+8k^2+2ij+4ik+8jk+i-j}}{(q^2;q^2)_i(q^2;q^2)_j(q^4;q^4)_k}
&=
\frac{(q^4;q^4)^3_\infty}{(q^2;q^2)_\infty(q^8;q^8)_\infty},
\label{table2.3.3}
\\
\sum_{i,j,k\ge 0}\frac{q^{i^2+3j^2+8k^2+2ij+4ik+8jk+i+j+4k}}{(q^2;q^2)_i(q^2;q^2)_j(q^4;q^4)_k}
&=\frac{(q^8;q^8)_\infty}{(q^2;q^2)_\infty}, \label{table2.3.4}
\\
\sum_{i,j,k\ge 0}\frac{q^{i^2+3j^2+8k^2+2ij+4ik+8jk+2i+j+4k}}{(q^2;q^2)_i(q^2;q^2)_j(q^4;q^4)_k}
&=\frac{1}{(q^3,q^4,q^5;q^8)_\infty}. \label{table2.3.5}
\end{align}

\begin{proof}
Summing over $j,k$ with $j+2k=m$ using Lemma \ref{lem-12} first, we have
\begin{align}
\sum_{i,j,k\ge 0}\frac{q^{i^2+3j^2+8k^2+2ij+4ik+8jk-j}u^iv^{j+2k}}{(q^2;q^2)_i(q^2;q^2)_j(q^4;q^4)_k}
=\sum_{i,m\ge 0}\frac{q^{i^2+2im+2m^2}u^iv^m}{(q^2;q^2)_i(q^2;q^2)_m}. \label{F2.3}
\end{align}
Setting $(u,v)=(1,1)$, $(q,q^{-2})$, $(q,1)$, $(q,q^2)$ and $(q^2,q^2)$ in  \eqref{F2.3},  and then using  \eqref{LW.rank2.3.5}--\eqref{LW.rank2.3.9}, we obtain \eqref{table2.3.1}--\eqref{table2.3.5}.
\end{proof}
\end{theorem}

\subsubsection{Example 4}
This example corresponds to
\begin{align*}
&A=\begin{pmatrix}
1 & 1 & 0\\
1 & 4 & 1\\
0 & 2 & 1
\end{pmatrix},\quad
AD=\begin{pmatrix}
    1 & 1 & 0\\
    1 & 4 & 2\\
    0 & 2 & 2
\end{pmatrix}
,\nonumber
\\
&b\in  \bigg\{
\begin{pmatrix}
1/2 \\ -2 \\ 1
\end{pmatrix},
\begin{pmatrix}
1/2 \\ 0 \\ 0
\end{pmatrix},
\begin{pmatrix}
1/2 \\ 0 \\ 1
\end{pmatrix},
\begin{pmatrix}
1/2 \\ 2 \\ 1
\end{pmatrix},
\begin{pmatrix}
1/2 \\ 2 \\ 2
\end{pmatrix}
\bigg\}.
\end{align*}
\begin{theorem}\label{thm-4}
We have
\begin{align}
\sum_{i,j,k\ge 0}\frac{q^{\frac{1}{2}i^2+2j^2+k^2+ij+2jk+\frac{1}{2}i-2j+k}}{(q;q)_i(q;q)_j(q^2;q^2)_k}
&=2\frac{(q^8;q^8)_\infty}{(q;q)_\infty}, \label{table2.4.1}
\\
\sum_{i,j,k\ge 0}\frac{q^{\frac{1}{2}i^2+2j^2+k^2+ij+2jk+\frac{1}{2}i}}{(q;q)_i(q;q)_j(q^2;q^2)_k}
&=\frac{(-q;q)_\infty}{(q,q^4,q^7;q^8)_\infty}, \label{table2.4.2}
\\
\sum_{i,j,k\ge 0}\frac{q^{\frac{1}{2}i^2+2j^2+k^2+ij+2jk+\frac{1}{2}i+k}}{(q;q)_i(q;q)_j(q^2;q^2)_k}
&=\frac{(q^4;q^4)^3_\infty}{(q;q)_\infty(q^2;q^2)_\infty(q^8;q^8)_\infty},
\label{table2.4.3}
\\
\sum_{i,j,k\ge 0}\frac{q^{\frac{1}{2}i^2+2j^2+k^2+ij+2jk+\frac{1}{2}i+2j+k}}{(q;q)_i(q;q)_j(q^2;q^2)_k}
&=\frac{(q^8;q^8)_\infty}{(q;q)_\infty},
\label{table2.4.4}
\\
\sum_{i,j,k\ge 0}\frac{q^{\frac{1}{2}i^2+2j^2+k^2+ij+2jk+\frac{1}{2}i+2j+2k}}{(q;q)_i(q;q)_j(q^2;q^2)_k}
&=\frac{(-q;q)_\infty}{(q^3,q^4,q^5;q^8)_\infty}.
\label{table2.4.5}
\end{align}
\end{theorem}
\begin{proof}
Summing over $i$ using \eqref{euler} first, we have
\begin{align}
\sum_{i,j,k\ge 0}\frac{q^{\frac{1}{2}i^2+2j^2+k^2+ij+2jk}u^iv^jw^k}{(q;q)_i(q;q)_j(q^2;q^2)_k}
=(-uq^\frac{1}{2};q)_\infty\sum_{j,k\ge 0}\frac{q^{2j^2+2jk+k^2}v^jw^k}{(q;q)_j(q^2;q^2)_k(-uq^\frac{1}{2};q)_j}. \label{F2.4}
\end{align}
Setting $(u,v,w)=(q^{\frac{1}{2}},q^{-2},q)$, $(q^{\frac{1}{2}},1,1)$, $(q^{\frac{1}{2}},1,q)$, $(q^{\frac{1}{2}},q^{2},q)$ and $(q^{\frac{1}{2}},q^2,q^2)$ in \eqref{F2.4}, by \eqref{LW.rank2.3.5}--\eqref{LW.rank2.3.9} we obtain \eqref{table2.4.1}--\eqref{table2.4.5}, respectively.
\end{proof}

\subsubsection{Example 9}
This example corresponds to
\begin{align*}
&A=\begin{pmatrix}
2 & 1 & 1\\
1 & 2 & 1\\
2 & 2 & 2
\end{pmatrix},\quad
AD=\begin{pmatrix}
    2 & 1 & 2\\
    1 & 2 & 2\\
    2 & 2 & 4
\end{pmatrix}
,\nonumber
\\
& b\in  \bigg\{
\begin{pmatrix}
-1 \\ 0 \\ 1
\end{pmatrix},
\begin{pmatrix}
-1/2 \\ 0 \\ 0
\end{pmatrix}, \begin{pmatrix}
0 \\ -1 \\ 1
\end{pmatrix}, \begin{pmatrix}
0 \\ -1/2 \\ 0
\end{pmatrix},
\begin{pmatrix}
0 \\ 0 \\ 1
\end{pmatrix},
\begin{pmatrix}
0 \\ 1 \\ 1
\end{pmatrix},
\begin{pmatrix}
1 \\ 0 \\ 1
\end{pmatrix},
\begin{pmatrix}
1/2 \\ 1 \\ 2
\end{pmatrix}, \begin{pmatrix}
1 \\ 1/2 \\2
\end{pmatrix}
\bigg\}.
\end{align*}
\begin{theorem}\label{thm-9}
We have
\begin{align}
\sum_{i,j,k\ge 0}\frac{q^{2i^2+2j^2+4k^2+2ij+4ik+4jk-2i+2k}}{(q^2;q^2)_i(q^2;q^2)_j(q^4;q^4)_k}
&=2\frac{(q^8;q^8)_\infty}{(q^2;q^2)_\infty}, \label{table2.9.1}
\\
\sum_{i,j,k\ge 0}\frac{q^{2i^2+2j^2+4k^2+2ij+4ik+4jk-j}}{(q^2;q^2)_i(q^2;q^2)_j(q^4;q^4)_k}
&=\frac{1}{(q,q^4,q^7;q^8)_\infty}, \label{table2.9.2}
\\
\sum_{i,j,k\ge 0}\frac{q^{2i^2+2j^2+4k^2+2ij+4ik+4jk+2k}}{(q^2;q^2)_i(q^2;q^2)_j(q^4;q^4)_k}
&=\frac{(q^4;q^4)^3_\infty}{(q^2;q^2)_\infty^2(q^8;q^8)_\infty},
\label{table2.9.5}
\\
\sum_{i,j,k\ge 0}\frac{q^{2i^2+2j^2+4k^2+2ij+4ik+4jk+2i+2k}}{(q^2;q^2)_i(q^2;q^2)_j(q^4;q^4)_k}
&=\frac{(q^8;q^8)_\infty}{(q^2;q^2)_\infty}, \label{table2.9.6}
\\
\sum_{i,j,k\ge 0}\frac{q^{2i^2+2j^2+4k^2+2ij+4ik+4jk+2i+j+4k}}{(q^2;q^2)_i(q^2;q^2)_j(q^4;q^4)_k}
&=\frac{1}{(q^3,q^4,q^5;q^8)_\infty}. \label{table2.9.8}
\end{align}
\end{theorem}
Interchanging $i$ with $j$, we obtain identities which confirm modularity of the remaining cases.
\begin{proof}
Summing over $j,k$ with $j+2k=m$ using Lemma \ref{lem-12} first, we have
\begin{align}
\sum_{i,j,k\ge 0}\frac{q^{2i^2+2j^2+4k^2+2ij+4ik+4jk-j}u^{i}v^{j+2k}}{(q^2;q^2)_i(q^2;q^2)_j(q^4;q^4)_k}
=\sum_{i,m\ge 0}\frac{q^{2i^2+2mi+m^2}u^iv^m}{(q^2;q^2)_i(q^2;q^2)_m}.
\label{F2.9}
\end{align}
Setting $(u,v)=(q^{-2},q)$, $(1,1)$, $(1,q)$, $(q^{2},q)$ and $(q^{2},q^2)$ in \eqref{F2.9}, by \eqref{LW.rank2.3.5}--\eqref{LW.rank2.3.9} we obtain  \eqref{table2.9.1}--\eqref{table2.9.8}, respectively.
\end{proof}

\subsection{Examples 5, 8, 14 and 17}
We will reduce the Nahm sums in these examples to the following double Nahm sums \cite{Wang-rank2}:
\begin{align}
&\sum_{i,j\ge 0}\frac{q^{2i^2+j^2+2ij}}{(q;q)_i(q;q)_j}
    =
    \frac{(q^3,q^4,q^7;q^7)_\infty}{(q;q)_\infty}, \quad \text{(\cite[Eq.\ (3.41)]{Wang-rank2})} \label{LW.rank2.3.41}
    \\
    &\sum_{i,j\ge 0}\frac{q^{2i^2+j^2+2ij+i}}{(q;q)_i(q;q)_j}
    =
    \frac{(q^2,q^5,q^7;q^7)_\infty}{(q;q)_\infty}, \quad \text{(\cite[Eq.\ (3.42)]{Wang-rank2})} \label{LW.rank2.3.42}
    \\
    &\sum_{i,j\ge 0}\frac{q^{2i^2+j^2+2ij+2i+j}}{(q;q)_i(q;q)_j}
    =
    \frac{(q,q^6,q^7;q^7)_\infty}{(q;q)_\infty}. \quad \text{(\cite[Eq.\ (3.43)]{Wang-rank2})} \label{LW.rank2.3.43}
\end{align}
\subsubsection{Example 5}
This example corresponds to
\begin{align*}
A=\begin{pmatrix}
1 & 1 & 0\\
1 & 4 & 2\\
0 & 4 & 4
\end{pmatrix},
\quad
AD=\begin{pmatrix}
    1 & 1 & 0\\
    1 & 4 & 4\\
    0 & 4 & 8
\end{pmatrix}
,\quad
b\in  \bigg\{
\begin{pmatrix}
1/2 \\ 0 \\ 0
\end{pmatrix},
\begin{pmatrix}
1/2 \\ 0 \\ 2
\end{pmatrix},
\begin{pmatrix}
1/2 \\ 2 \\ 4
\end{pmatrix}
\bigg\}.
\end{align*}
\begin{theorem}\label{thm-5}
We have
\begin{align}
\sum_{i,j,k\ge 0}\frac{q^{\frac{1}{2}i^2+2j^2+4k^2+ij+4jk+\frac{1}{2}i}}{(q;q)_i(q;q)_j(q^2;q^2)_k}
&=\frac{(q^6,q^8,q^{14};q^{14})_\infty}{(q;q)_\infty}, \label{table2.5.1}
\\
\sum_{i,j,k\ge 0}\frac{q^{\frac{1}{2}i^2+2j^2+4k^2+ij+4jk+\frac{1}{2}i+2k}}{(q;q)_i(q;q)_j(q^2;q^2)_k}
&=\frac{(q^4,q^{10},q^{14};q^{14})_\infty}{(q;q)_\infty}, \label{table2.5.2}
\\
\sum_{i,j,k\ge 0}\frac{q^{\frac{1}{2}i^2+2j^2+4k^2+ij+4jk+\frac{1}{2}i+2j+4k}}{(q;q)_i(q;q)_j(q^2;q^2)_k}
&=\frac{(q^2,q^{12},q^{14};q^{14})_\infty}{(q;q)_\infty}. \label{table2.5.3}
\end{align}
\end{theorem}
\begin{proof}
Summing over $i$ using \eqref{euler} first, we have
\begin{align}
\sum_{i,j,k\ge 0}\frac{q^{\frac{1}{2}i^2+2j^2+4k^2+ij+4jk}u^iv^jw^k}{(q;q)_i(q;q)_j(q^2;q^2)_k}
=(-uq^\frac{1}{2};q)_\infty\sum_{j,k\ge 0}\frac{q^{2j^2+4jk+4k^2}v^jw^k}{(q;q)_j(q^2;q^2)_k(-uq^\frac{1}{2};q)_j}. \label{F2.5}
\end{align}
Setting $(u,v,w)=(q^{\frac{1}{2}},1,1)$, $(q^\frac{1}{2},1,q^2)$ and $(q^{\frac{1}{2}},q^2,q^4)$ in \eqref{F2.5}, by \eqref{LW.rank2.3.41}--\eqref{LW.rank2.3.43} we obtain \eqref{table2.5.1}--\eqref{table2.5.3}, respectively.
\end{proof}

\subsubsection{Example 8}
This example corresponds to
\begin{align*}
A=\begin{pmatrix}
1 & 1 & 0\\
1 & 8 & 2\\
0 & 4 & 2
\end{pmatrix},\quad
AD=\begin{pmatrix}
    1 & 1 & 0\\
    1 & 8 & 4\\
    0 & 4 & 4
\end{pmatrix}
,\quad
b\in  \bigg\{
\begin{pmatrix}
1/2 \\ 0 \\ 0
\end{pmatrix},
\begin{pmatrix}
1/2 \\ 2 \\ 0
\end{pmatrix},
\begin{pmatrix}
1/2 \\ 4 \\ 2
\end{pmatrix}
\bigg\}.
\end{align*}
\begin{theorem}\label{thm-8}
We have
\begin{align}
\sum_{i,j,k\ge 0}\frac{q^{\frac{1}{2}i^2+4j^2+2k^2+ij+4jk+\frac{1}{2}i}}{(q;q)_i(q;q)_j(q^2;q^2)_k}
&=\frac{(q^6,q^8,q^{14};q^{14})_\infty}{(q;q)_\infty}, \label{table2.8.1}
\\
\sum_{i,j,k\ge 0}\frac{q^{\frac{1}{2}i^2+4j^2+2k^2+ij+4jk+\frac{1}{2}i+2j}}{(q;q)_i(q;q)_j(q^2;q^2)_k}
&=\frac{(q^4,q^{10},q^{14};q^{14})_\infty}{(q;q)_\infty}, \label{table2.8.2}
\\
\sum_{i,j,k\ge 0}\frac{q^{\frac{1}{2}i^2+4j^2+2k^2+ij+4jk+\frac{1}{2}i+4j+2k}}{(q;q)_i(q;q)_j(q^2;q^2)_k}
&=\frac{(q^2,q^{12},q^{14};q^{14})_\infty}{(q;q)_\infty}.  \label{table2.8.3}
\end{align}
\end{theorem}
\begin{proof}
Summing over $i$ using \eqref{euler} first, we have
\begin{align}
\sum_{i,j,k\ge 0}\frac{q^{\frac{1}{2}i^2+4j^2+2k^2+ij+4jk}u^iv^jw^k}{(q;q)_i(q;q)_j(q^2;q^2)_k}
=(-uq^{\frac{1}{2}};q)_\infty\sum_{j,k\ge 0}\frac{q^{4j^2+4jk+2k^2}v^jw^k}{(q;q)_j(q^2;q^2)_k(-uq^\frac{1}{2};q)_j}. \label{F2.8}
\end{align}
Setting $(u,v,w)=(q^\frac{1}{2},1,1)$, $(q^\frac{1}{2},q^2,1)$ and $(q^\frac{1}{2},q^4,q^2)$ in \eqref{F2.8}, and then using  \eqref{LW.rank2.3.41}--\eqref{LW.rank2.3.43}, we obtain \eqref{table2.8.1}--\eqref{table2.8.3}, respectively.
\end{proof}

\subsubsection{Example 14}
This example corresponds to
\begin{align*}
A=\begin{pmatrix}
2 & 2 & 2\\
2 & 5 & 4\\
4 & 8 & 8
\end{pmatrix},\quad
AD=\begin{pmatrix}
2 & 2 & 4\\
2 & 5 & 8\\
4 & 8 & 16
\end{pmatrix},\quad
b\in  \bigg\{
\begin{pmatrix}
0 \\-1/2\\0
\end{pmatrix},
\begin{pmatrix}
0 \\1/2\\2
\end{pmatrix},
\begin{pmatrix}
1 \\3/2\\4
\end{pmatrix}
\bigg\}.
\end{align*}
\begin{theorem}\label{thm-14}
We have
\begin{align}
\sum_{i,j,k\ge 0}\frac{q^{i^2+\frac{5}{2}j^2+8k^2+2ij+4ik+8jk-\frac{1}{2}j}}{(q;q)_i(q;q)_j(q^2;q^2)_k}
&=\frac{(q^3,q^4,q^7;q^7)_\infty}{(q;q)_\infty}, \label{table2.14.1}
\\
\sum_{i,j,k\ge 0}\frac{q^{i^2+\frac{5}{2}j^2+8k^2+2ij+4ik+8jk+\frac{1}{2}j+2k}}{(q;q)_i(q;q)_j(q^2;q^2)_k}
&=\frac{(q^2,q^5,q^7;q^7)_\infty}{(q;q)_\infty}, \label{table2.14.2}
\\
\sum_{i,j,k\ge 0}\frac{q^{i^2+\frac{5}{2}j^2+8k^2+2ij+4ik+8jk+i+\frac{3}{2}j+4k}}{(q;q)_i(q;q)_j(q^2;q^2)_k}
&=
\frac{(q,q^6,q^7;q^7)_\infty}{(q;q)_\infty}. \label{table2.14.3}
\end{align}
\end{theorem}
\begin{proof}
Summing over $j,k$ with $j+2k=m$ using Lemma \ref{lem-12} first, we have
\begin{align}
\sum_{i,j,k\ge 0}\frac{q^{i^2+\frac{5}{2}j^2+8k^2+2ij+4ik+8jk-\frac{1}{2}j}u^iv^{j+2k}}{(q;q)_i(q;q)_j(q^2;q^2)_k}
=\sum_{i,m\ge 0}\frac{q^{i^2+2im+2m^2}u^iv^m}{(q;q)_i(q;q)_m}.
\label{F2.14}
\end{align}
Setting $(u,v)=(1,1)$, $(1,q)$ and $(q,q^2)$ in \eqref{F2.14}, and then using \eqref{LW.rank2.3.41}--\eqref{LW.rank2.3.43}, we obtain \eqref{table2.14.1}--\eqref{table2.14.3}, respectively.
\end{proof}

\subsubsection{Example 17}
This example corresponds to
\begin{align*}
A=\begin{pmatrix}
3 & 2 & 2\\
2 & 4 & 2\\
4 & 4 & 4
\end{pmatrix},
\quad
AD=\begin{pmatrix}
    3 & 2 & 4\\
    2 & 4 & 4\\
    4 & 4 & 8
\end{pmatrix}
,\quad
b\in  \bigg\{
\begin{pmatrix}
-1/2 \\ 0 \\ 0
\end{pmatrix},
\begin{pmatrix}
-1/2 \\ 1 \\ 0
\end{pmatrix},
\begin{pmatrix}
1/2 \\ 2 \\ 2
\end{pmatrix}
\bigg\}.
\end{align*}
\begin{theorem}\label{thm-17}
We have
\begin{align}
\sum_{i,j,k\ge 0}\frac{q^{\frac{3}{2}i^2+2j^2+4k^2+2ij+4ik+4jk-\frac{1}{2}i}}{(q;q)_i(q;q)_j(q^2;q^2)_k}
&=\frac{(q^3,q^4,q^7;q^7)_\infty}{(q;q)_\infty},
\label{table2.17.1}
\\
\sum_{i,j,k\ge 0}\frac{q^{\frac{3}{2}i^2+2j^2+4k^2+2ij+4ik+4jk-\frac{1}{2}i+j}}{(q;q)_i(q;q)_j(q^2;q^2)_k}
&=\frac{(q^2,q^5,q^7;q^7)_\infty}{(q;q)_\infty},
\label{table2.17.2}
\\
\sum_{i,j,k\ge 0}\frac{q^{\frac{3}{2}i^2+2j^2+4k^2+2ij+4ik+4jk+\frac{1}{2}i+2j+2k}}{(q;q)_i(q;q)_j(q^2;q^2)_k}
&=\frac{(q,q^6,q^7;q^7)_\infty}{(q;q)_\infty}.
\label{table2.17.3}
\end{align}
\end{theorem}
\begin{proof}
Summing over $i,k$ with $i+2k=m$ using Lemma \ref{lem-12} first, we have
\begin{align}
\sum_{i,j,k\ge 0}\frac{q^{\frac{3}{2}i^2+2j^2+4k^2+2ij+4ik+4jk-\frac{1}{2}i}u^{i+2k}v^j}{(q;q)_i(q;q)_j(q^2;q^2)_k}
=\sum_{j,m\ge 0}\frac{q^{2j^2+2jm+m^2}u^mv^j}{(q;q)_j(q;q)_m}.
\label{F2.17}
\end{align}
Setting $(u,v) = (1,1)$, $(1,q)$ and $(q,q^2)$ in \eqref{F2.17}, and then using \eqref{LW.rank2.3.41}--\eqref{LW.rank2.3.43}, we obtain \eqref{table2.17.1}--\eqref{table2.17.3}, respectively.
\end{proof}

\subsection{Example 11}
This example corresponds to
\begin{align*}
&A=\begin{pmatrix}
2 & 2 & 1\\
2 & 4 & 2\\
2 & 4 & 3
\end{pmatrix},\quad
AD=\begin{pmatrix}
    2 & 2 & 2\\
    2 & 4 & 4\\
    2 & 4 & 6
\end{pmatrix}
,\nonumber
\\
&b\in  \bigg\{
\begin{pmatrix}
-1\\ -1\\-1
\end{pmatrix},
\begin{pmatrix}
0 \\ 0 \\ 0
\end{pmatrix},
\begin{pmatrix}
0 \\ 0 \\ 1
\end{pmatrix},
\begin{pmatrix}
0 \\ 1 \\ 2
\end{pmatrix},
\begin{pmatrix}
1 \\ 2 \\ 3
\end{pmatrix}
\bigg\}.
\end{align*}
Its modularity follows from the following known identities:
\begin{align}
    \sum_{i,j,k\ge 0}\frac{q^{i^2+2j^2+3k^2+2ij+2ik+4jk-i-j-k}}{(q;q)_i(q;q)_j(q^2;q^2)_k}
    &=2\frac{(q^3,q^5,q^8;q^8)_\infty}{(q;q)_\infty}, \label{proof-11-1}
    \\
    \sum_{i,j,k\ge 0}\frac{q^{i^2+2j^2+3k^2+2ij+2ik+4jk}}{(q;q)_i(q;q)_j(q^2;q^2)_k}
    &=\frac{(q^4,q^4,q^8;q^8)_\infty}{(q;q)_\infty},
    \\
    \sum_{i,j,k\ge 0}\frac{q^{i^2+2j^2+3k^2+2ij+2ik+4jk+k}}{(q;q)_i(q;q)_j(q^2;q^2)_k}
    &=\frac{(q^3,q^5,q^8;q^8)_\infty}{(q;q)_\infty},
    \\
    \sum_{i,j,k\ge 0}\frac{q^{i^2+2j^2+3k^2+2ij+2ik+4jk+j+2k}}{(q;q)_i(q;q)_j(q^2;q^2)_k}
    &=
    \frac{(q^2,q^6,q^8;q^8)_\infty}{(q;q)_\infty},
    \\
     \sum_{i,j,k\ge 0}\frac{q^{i^2+2j^2+3k^2+2ij+2ik+4jk+i+2j+3k}}{(q;q)_i(q;q)_j(q^2;q^2)_k}
    &=\frac{(q,q^7,q^8;q^8)_\infty}{(q;q)_\infty}.
\end{align}
The first identity follows as a special case of the following identity:
\begin{align}
    \sum_{n_1,\cdots,n_{k-1}\geq 0} \frac{q^{N_1^2+\cdots+N_{k-1}^2-N_1}}{(q;q)_{n_1}\cdots (q;q)_{n_{k-2}} (q^2;q^2)_{n_{k-1}}}  =2\frac{(q^{k-1},q^{k+1},q^{2k};q^{2k})_\infty}{(q;q)_\infty},
\end{align}
which corresponds to the case $i=1$ of the following identity of Bressoud \cite{Bressoud1980} (see also \cite[Eq.\ (2.8)]{DJK}):
\begin{align}
   & \sum_{n_1,\dots,n_{k-1}\geq 0} \frac{q^{N_1^2+\cdots+N_{k-1}^2-N_1-\cdots-N_i}}{(q;q)_{n_1}\cdots (q;q)_{n_{k-2}}(q^2;q^2)_{n_{k-1}}} \nonumber \\
   &=\sum_{m=0}^i\frac{(q^{2k},q^{k-i+2m},q^{k+i-2m};q^{2k})_\infty}{(q;q)_\infty}.
\end{align}
The last four identities  are special instances of Bressoud's identity \eqref{eq-Bressoud}.

\subsection{Examples 12 and 13}
\subsubsection{Example 12}
This example corresponds to
\begin{align*}
A=\begin{pmatrix}
1 & 1 & 1/2\\
1 & 2 & 1\\
1 & 2 & 3/2
\end{pmatrix}, \quad
AD=\begin{pmatrix}
    1 & 1 & 1\\
    1 & 2 & 2\\
    1 & 2 & 3
\end{pmatrix}
,\quad
b\in  \bigg\{
\begin{pmatrix}
0 \\ 0 \\ 0
\end{pmatrix},
\begin{pmatrix}
0 \\ 0 \\ 1
\end{pmatrix},
\begin{pmatrix}
0 \\ 1 \\ 1
\end{pmatrix},
\begin{pmatrix}
1 \\ 1 \\ 2
\end{pmatrix}
\bigg\}.
\end{align*}
\begin{theorem}\label{thm-12}
We have
\begin{align}
    \sum_{i,j,k\ge 0}\frac{q^{i^2+2j^2+3k^2+2ij+2ik+4jk}}{(q^2;q^2)_i(q^2;q^2)_j(q^4;q^4)_k}
    &=\frac{(q^4,q^5,q^9;q^9)_\infty}
    {(q;q^2)_\infty(q^4;q^4)_\infty}, \label{exam12-1}
    \\
    \sum_{i,j,k\ge 0}\frac{q^{i^2+2j^2+3k^2+2ij+2ik+4jk+2k}}{(q^2;q^2)_i(q^2;q^2)_j(q^4;q^4)_k}
    &=\frac{(q^3;q^3)_\infty}
    {(q;q^2)_\infty(q^4;q^4)_\infty}, \label{exam12-2}
    \\
    \sum_{i,j,k\ge 0}\frac{q^{i^2+2j^2+3k^2+2ij+2ik+4jk+2j+2k}}{(q^2;q^2)_i(q^2;q^2)_j(q^4;q^4)_k}
    &=\frac{(q^2,q^7,q^9;q^9)_\infty}
    {(q;q^2)_\infty(q^4;q^4)_\infty}, \label{exam12-3}
    \\
    \sum_{i,j,k\ge 0}\frac{q^{i^2+2j^2+3k^2+2ij+2ik+4jk+2i+2j+4k}}{(q^2;q^2)_i(q^2;q^2)_j(q^4;q^4)_k}
    &=\frac{(q,q^8,q^9;q^9)_\infty}
    {(q;q^2)_\infty(q^4;q^4)_\infty}. \label{exam12-4}
\end{align}
\end{theorem}
\begin{proof}
Let
\begin{align}\label{exam12-start-F}
&    F(u,w,w)=F(u,v,w;q):=\sum_{i,j,k\geq 0} \frac{q^{i^2+2j^2+3k^2+2ij+2ik+4jk}u^iv^jw^k}{(q^2;q^2)_i(q^2;q^2)_j(q^4;q^4)_k} \nonumber \\
&=\sum_{j,k\geq 0} \frac{q^{2(j+k)^2+k^2}v^jw^k}{(q^2;q^2)_j(q^4;q^4)_k}(-q^{2j+2k+1}u;q^2)_\infty.
\end{align}
In particular, when $v=w$ we have
\begin{align}\label{exam12-start-F-special}
    F(u,v,v)=(-qu;q^2)_\infty \sum_{n=0}^\infty \sum_{j+k=n} \frac{q^{2n^2+k^2}v^n}{(q^2;q^2)_{j}(q^4;q^4)_k(-qu;q^2)_n}.
\end{align}

(1) Applying \eqref{BL3} with $a=1$ to the \eqref{B-G(1)} Bailey pair $(\alpha_n(1;q),\beta_n(1;q))$, we obtain a new Bailey pair $(\alpha_n'(1;q),\beta_n'(1;q))$ with
\begin{align}
    \alpha_n'(1;q)=q^{n^2/2}\alpha_n(1;q), \quad \beta_n'(1;q)=\frac{1}{(-q^{1/2};q)_n}\sum_{k=0}^n\frac{q^{k^2/2}}{(q;q)_{n-k}(q^2;q^2)_k}.
\end{align}
By \eqref{exam12-start-F-special} we have
\begin{align}\label{exam12-1-proof}
    F(1,1,1)=(-q;q^2)_\infty \sum_{n=0}^\infty q^{2n^2}\beta_n'(1;q^2).
\end{align}
Replacing $q$ by $q^2$ in \eqref{cor1} and setting $a=1$, we deduce that
\begin{align}
    &\sum_{n=0}^\infty q^{2n^2}\beta_n'(1;q^2)=\frac{1}{(q^2;q^2)_\infty} \sum_{r=0}^\infty q^{2r^2}\alpha_r'(1;q^2)
    =\frac{1}{(q^2;q^2)_\infty} \sum_{r=0}^\infty q^{3r^2}\alpha_r(1;q^2) \nonumber \\
    &=\frac{1}{(q^2;q^2)_\infty}  \sum_{r=-\infty}^\infty (-1)^rq^{(9r^2-r)/2}  =\frac{(q^4,q^5,q^9;q^9)_\infty}{(q^2;q^2)_\infty}.
\end{align}
Here  for the last equality we used \eqref{Jacobi}.
Substituting it into \eqref{exam12-1-proof} we obtain \eqref{exam12-1}.

(2) Applying \eqref{BL3} with $a=1$ to the \eqref{G(3)} Bailey pair $(\alpha_n(1;q),\beta_n(1;q))$, we obtain a new Bailey pair:
\begin{align}
    \alpha_n'(1;q)=q^{n^2/2}\alpha_n(1;q), \quad \beta_n'(1;q)=\frac{1}{(-q^{1/2};q)_n}\sum_{k=0}^n\frac{q^{k^2/2+k}}{(q;q)_{n-k}(q^2;q^2)_k}.
\end{align}
By \eqref{exam12-start-F} we have
\begin{align}\label{exam12-2-proof}
     F(1,1,q^2)&=(-q;q^2)_\infty \sum_{n=0}^\infty \frac{q^{2n^2}}{(-q;q^2)_n}\sum_{k=0}^n \frac{q^{k^2+2k}}{(q^4;q^4)_k(q^2;q^2)_{n-k}} \nonumber \\
     &=(-q;q^2)_\infty\sum_{n=0}^\infty q^{2n^2}\beta_n'(1;q^2).
\end{align}
Using \eqref{cor1} with $a=1$ and $q$ replaced by $q^2$, we have
\begin{align}
    &\sum_{n=0}^\infty q^{2n^2}\beta_n'(1;q^2) =\frac{1}{(q^2;q^2)_\infty} \sum_{r=0}^\infty q^{2r^2}\alpha_r'(1;q^2) =\frac{1}{(q^2;q^2)_\infty} \sum_{r=0}^\infty q^{3r^2}\alpha_r(1;q^2) \nonumber \\
    &=\frac{1}{(q^2;q^2)_\infty} \sum_{r=-\infty}^\infty (-1)^rq^{(9r^2-3r)/2}=\frac{(q^3;q^3)_\infty}{(q^2;q^2)_\infty}.
\end{align}
Here for the last equality we used \eqref{Jacobi}. Substituting it into \eqref{exam12-2-proof}, we obtain \eqref{exam12-2}.

(3) Applying \eqref{BL5} with $a=q$ to the \eqref{G(2)} Bailey pair $(\alpha_n(q;q),\beta_n(q;q))$, we obtain a new Bailey pair:
\begin{equation}
\begin{split}
   & \alpha_n'(q;q)=\frac{(-q^{3/2};q)_n}{(-q^{1/2};q)_n}q^{n^2/2}\alpha_n(q;q), \\
    & \beta_n'(q;q)=\frac{1}{(-q^{1/2};q)_n}\sum_{k=0}^n \frac{q^{k^2/2}}{(q;q)_{n-k}(q^2;q^2)_k}.
\end{split}
\end{equation}
By \eqref{exam12-start-F-special} we have
\begin{align}\label{exam12-3-proof}
     F(1,q^2,q^2)=(-q;q^2)_\infty \sum_{n=0}^\infty q^{2n^2+2n}\beta_n'(q^2;q^2).
\end{align}
Using \eqref{cor1} with $q$ replaced by $q^2$ we have
\begin{align*}
   &\sum_{n=0}^\infty q^{2n^2+2n}\beta_n'(q^2;q^2) =\frac{1}{(q^4;q^2)_\infty} \sum_{r=0}^\infty q^{2r^2+2r}\frac{(-q^3;q^2)_r}{(-q;q^2)_r}q^{r^2}\alpha_r(q^2;q^2) \nonumber \\
   &=\frac{1}{(q^4;q^2)_\infty} \sum_{r=0}^\infty (-1)^r q^{\frac{9}{2}r^2+\frac{5}{2}r} \cdot \frac{1-q^{4r+2}}{1-q^2} =\frac{1}{(q^2;q^2)_\infty} \sum_{r=-\infty}^\infty (-1)^rq^{\frac{9}{2}r^2+\frac{5}{2}r} \nonumber \\
   &=\frac{(q^2,q^7,q^9;q^9)_\infty}{(q^2;q^2)_\infty}.
\end{align*}
Here for the last equality we used \eqref{Jacobi}. Substituting it into \eqref{exam12-3-proof}, we obtain \eqref{exam12-3}.

(4) Applying \eqref{BL3} with $a=q$ to the \eqref{B-G(1.1)} Bailey pair, we obtain the new Bailey pair:
\begin{equation}
\begin{split}
    &\alpha_n'(q;q)=\frac{1+q^{1/2}}{1+q^{n+1/2}}q^{\frac{1}{2}n^2+n}\alpha_n(q;q), \\
    & \beta_n'(q;q)=\frac{1}{(-q^{3/2};q)_n}\sum_{k=0}^n \frac{(-q^{1/2};q)_k}{(q;q)_{n-k}(q^2;q^2)_k}q^{\frac{1}{2}k^2+k}.
\end{split}
\end{equation}

Using \eqref{exam12-start-F}, \eqref{cor1} and this Bailey pair with $q$ replaced by $q^2$,
we have
\begin{align*}
    &F(q^2,q^2,q^4)=(-q^3;q^2)_\infty \sum_{n=0}^\infty \frac{q^{2n^2+2n}}{(-q^3;q^2)_n} \sum_{k=0}^n\frac{q^{k^2+2k}}{(q^4;q^4)_k(q^2;q^2)_{n-k}} \nonumber \\
    &=(-q^3;q^2)_\infty \sum_{n=0}^\infty q^{2n^2+2n}\beta_n'(q^2;q^2) \nonumber \\
    &=\frac{(-q^3;q^2)_\infty}{(q^4;q^2)_\infty} \sum_{r=0}^\infty q^{3r^2+4r}\cdot \frac{1+q}{1+q^{2r+1}}\alpha_r(q^2;q^2) \nonumber \\
   &=\frac{(-q;q^2)_\infty}{(q^2;q^2)_\infty} \sum_{r=0}^\infty (-1)^rq^{\frac{9}{2}r^2+\frac{7}{2}r}(1-q^{2r+1}) \\
   &=\frac{(-q;q^2)_\infty}{(q^2;q^2)_\infty} \sum_{r=-\infty}^\infty (-1)^rq^{\frac{9}{2}r^2+\frac{7}{2}r} \nonumber \\
    &=\frac{(-q;q^2)_\infty (q,q^8,q^9;q^9)_\infty}{(q^2;q^2)_\infty}.
\end{align*}
Here for the last equality we used \eqref{Jacobi}. This proves \eqref{exam12-4}.
\end{proof}

\begin{rem}
    Note that \eqref{exam12-1} is a special case of Warnaar's identity \cite[Eq.\ (5.14)]{Warnaar}:
\begin{align}
&\sum_{n_1,\dots,n_k\geq 0} \frac{q^{\frac{1}{2}(N_1^2+\cdots+N_k^2)}}{(q;q)_{n_1}\cdots (q;q)_{n_{k-1}}(q^{2};q^{2})_{n_k}} \nonumber \\
&=\frac{(q^{(k+1)/2},q^{(k+2)/2},q^{(2k+3)/2};q^{(2k+3)/2})_\infty}{(-q;q)_\infty(q^{1/2};q^{1/2})_\infty}
\end{align}
where $N_i=n_i+\cdots +n_k$.
Warnaar also discovered a more general identity (unpublished, see \eqref{eq-Warnaar-general}) which covers the identities \eqref{exam12-1}--\eqref{exam12-4}.
\end{rem}

\subsubsection{Example 13}
This example corresponds to
\begin{align*}
A=\begin{pmatrix}
1 & 0 & 1/2\\
0 & 2 & 1\\
1 & 2 & 2
\end{pmatrix},
\quad
AD=\begin{pmatrix}
    1 & 0 & 1\\
    0 & 2 & 2\\
    1 & 2 & 4
\end{pmatrix}
,\quad
b\in  \bigg\{
\begin{pmatrix}
0 \\ 0 \\ 0
\end{pmatrix},
\begin{pmatrix}
0 \\ 0 \\ 1
\end{pmatrix},
\begin{pmatrix}
0 \\ 1 \\ 2
\end{pmatrix},
\begin{pmatrix}
1 \\ 1 \\ 2
\end{pmatrix}
\bigg\}.
\end{align*}
\begin{theorem}\label{thm-13}
We have
\begin{align}
    \sum_{i,j,k\ge 0}\frac{q^{i^2+2j^2+4k^2+2ik+4jk}}{(q^2;q^2)_i(q^2;q^2)_j(q^4;q^4)_k}
    &=
    \frac{(q^5,q^6,q^{11};q^{11})_\infty}{(q;q^2)_\infty(q^4;q^4)_\infty},\label{table2.13.1}
    \\
    \sum_{i,j,k\ge 0}\frac{q^{i^2+2j^2+4k^2+2ik+4jk+2k}}{(q^2;q^2)_i(q^2;q^2)_j(q^4;q^4)_k}
    &=
    \frac{(q^4,q^7,q^{11};q^{11})_\infty}{(q;q^2)_\infty(q^4;q^4)_\infty},\label{table2.13.2}
    \\\sum_{i,j,k\ge 0}\frac{q^{i^2+2j^2+4k^2+2ik+4jk+2j+4k}}{(q^2;q^2)_i(q^2;q^2)_j(q^4;q^4)_k}
    &=
    \frac{(q^2,q^9,q^{11};q^{11})_\infty}{(q;q^2)_\infty(q^4;q^4)_\infty},\label{table2.13.3}
    \\\sum_{i,j,k\ge 0}\frac{q^{i^2+2j^2+4k^2+2ik+4jk+2i+2j+4k}}{(q^2;q^2)_i(q^2;q^2)_j(q^4;q^4)_k}
    &=
    \frac{(q,q^{10},q^{11};q^{11})_\infty}{(q;q^2)_\infty(q^4;q^4)_\infty}.\label{table2.13.4}
\end{align}
\end{theorem}
This theorem can be proved in a way similar to the proof of Theorem \ref{thm-12}. We can generalize it to multi-sum identities. See Section \ref{sec-general} for its generalization and complete proofs.

Below we give a companion identity which includes the product $(q^3,q^8,q^{11};q^{11})_\infty$ missed from Theorem \ref{thm-13}.
\begin{theorem}\label{thm-13-sum}
    We have
    \begin{align}\label{eq-13-sum}
    \sum_{i,j,k\ge 0}\frac{q^{i^2+2j^2+4k^2+2ik+4jk+2j+2k}(1+q^{2i+2j+4k+2})}{(q^2;q^2)_i(q^2;q^2)_j(q^4;q^4)_k}
    =
    \frac{(q^3,q^8,q^{11};q^{11})_\infty}{(q;q^2)_\infty(q^4;q^4)_\infty},
    \end{align}
\end{theorem}
\begin{proof}
We need the following special case of the $q$-Chu--Vandermonde summation formula \cite[p.\ 20]{Andrews1984}:
\begin{align}
\frac{1}{(q;q)_i(q;q)_j}=\sum_{l=0}^{\min(i,j)} \frac{q^{(i-l)(j-l)}}{(q;q)_l(q;q)_{i-l}(q;q)_{j-l}}. \label{Andrews-id}
\end{align}
Using this identity, we have
\begin{align}
    &\sum_{i,j,k\ge 0}\frac{q^{i^2+2j^2+4k^2+2ik+4jk+2j+2k}(1+q^{2i+2j+4k+2})}{(q^2;q^2)_i(q^2;q^2)_j(q^4;q^4)_k} \nonumber \\
    &=\sum_{i,j,k,l\ge 0}\frac{q^{(i+l)^2+2(j+l)^2+4k^2+2(i+l)k+4(j+l)k+2ij+2(j+l)+2k}(1+q^{2(i+l)+2(j+l)+4k+2})}{(q^2;q^2)_i(q^2;q^2)_j(q^2;q^2)_l(q^4;q^4)_k}\nonumber \\
    &=\sum_{i,j,k,l\ge 0}\frac{q^{k^2+(k+l)^2+(k+l+j)^2+(k+l+j+i)^2+2j+2l+2k}(1+q^{2(i+l)+2(j+l)+4k+2})}{(q^2;q^2)_i(q^2;q^2)_j(q^2;q^2)_l(q^4;q^4)_k}\nonumber \\
    &=\sum_{i,j,k,l\ge 0}\frac{q^{k^2+(k+l)^2+(k+l+j)^2+(k+l+j+i)^2+2k+2l}}{(q^2;q^2)_i(q^2;q^2)_j(q^2;q^2)_l(q^4;q^4)_k}. \label{exam13-proof-last}
\end{align}
Now using Warnaar's identity \eqref{eq-Warnaar-general} with $(k,i)=(4,3)$ and $q$ replaced by $q^2$, we obtain \eqref{eq-13-sum}.
\end{proof}

\begin{rem}
We give more details on the last step in \eqref{exam13-proof-last}. Let $A(j)$ be any sequence such that
\begin{align}
    G(u):=\sum_{j\ge 0}\frac{u^jA(j)}{(q^2;q^2)_j}
\end{align}
converges. Then it is easy to see that
\begin{align}\label{G-diff}
G(u)-G(uq^2)=\sum_{j\ge 0}\frac{u^{j}(1-q^{2j})A(j)}{(q^2;q^2)_j}=\sum_{j\ge 0}\frac{u^{j+1}A(j+1)}{(q^2;q^2)_j}  .
\end{align}
If we set
\begin{align*}
    A(j)=\frac{q^{k^2+(k+l)^2+(k+l+j)^2+(k+l+j+i)^2+2k+2l}}{(q^2;q^2)_i(q^2;q^2)_j(q^2;q^2)_k(q^2;q^2)_l},
\end{align*}
then we obtain the last equality in \eqref{exam13-proof-last} immediately by setting $u=1$ in \eqref{G-diff}.
\end{rem}

\subsection{Example 15}
This example corresponds to
\begin{align*}
A=\begin{pmatrix}
1 & 1 & 0\\
1 & 3 & 1/2\\
0 & 1 & 3/2
\end{pmatrix},
\quad
AD=\begin{pmatrix}
    1 & 1 & 0\\
    1 & 3 & 1\\
    0 & 1 & 3
\end{pmatrix}
, \nonumber \\
b\in  \bigg\{
\begin{pmatrix}
1/2 \\ -1/2 \\ 1/2
\end{pmatrix},
\begin{pmatrix}
1/2 \\ 1/2 \\-1/2
\end{pmatrix},
\begin{pmatrix}
1/2 \\ 1/2 \\ 3/2
\end{pmatrix},
\begin{pmatrix}
1/2 \\ 3/2 \\ 1/2
\end{pmatrix}
\bigg\}.
\end{align*}
\begin{theorem}\label{thm-15}
We have
\begin{align}
\sum_{i,j,k\ge 0}\frac{q^{\frac{1}{2}i^2+\frac{3}{2}j^2+\frac{3}{2}k^2+ij+jk+\frac{1}{2}i-\frac{1}{2}j+\frac{1}{2}k}}{(q;q)_i(q;q)_j(q^2;q^2)_k}
&=\frac{(-q;q)_\infty}{(q,q^4;q^5)_\infty},
\label{table2.15.1}
\\
\sum_{i,j,k\ge 0}\frac{q^{\frac{1}{2}i^2+\frac{3}{2}j^2+\frac{3}{2}k^2+ij+jk+\frac{1}{2}i+\frac{1}{2}j-\frac{1}{2}k}}{(q;q)_i(q;q)_j(q^2;q^2)_k}
&=\frac{(-q;q)_\infty}{(q,q^4;q^5)_\infty},
\label{table2.15.2}
\\
\sum_{i,j,k\ge 0}\frac{q^{\frac{1}{2}i^2+\frac{3}{2}j^2+\frac{3}{2}k^2+ij+jk+\frac{1}{2}i+\frac{1}{2}j+\frac{3}{2}k}}{(q;q)_i(q;q)_j(q^2;q^2)_k}
&=\frac{(-q;q)_\infty}{(q^2,q^3;q^5)_\infty},
\label{table2.15.3}
\\
\sum_{i,j,k\ge 0}\frac{q^{\frac{1}{2}i^2+\frac{3}{2}j^2+\frac{3}{2}k^2+ij+jk+\frac{1}{2}i+\frac{3}{2}j+\frac{1}{2}k}}{(q;q)_i(q;q)_j(q^2;q^2)_k}
&=\frac{(-q;q)_\infty}{(q^2,q^3;q^5)_\infty}.
\label{table2.15.4}
\end{align}
\end{theorem}
\begin{proof}
Summing over $i$ using \eqref{euler} first, we have
\begin{align}
&F(u,v,w;q)
:=\sum_{i,j,k\ge 0}\frac{q^{\frac{1}{2}i^2+\frac{3}{2}j^2+\frac{3}{2}k^2+ij+jk}u^iv^jw^k}{(q;q)_i(q;q)_j(q^2;q^2)_k} \nonumber \\
&=
(-uq^{\frac{1}{2}};q)_\infty\sum_{j,k\ge 0}\frac{q^{\frac{3}{2}j^2+jk+\frac{3}{2}k^2}v^jw^k}{(q;q)_j(q^2;q^2)_k(-uq^{\frac{1}{2}};q)_j}.\label{F2.15}
\end{align}

Setting $(u,v,w)=(q^{\frac{1}{2}},q^{a-\frac{1}{2}},q^{a+\frac{1}{2}})$ ($a=0,1$), we have by \eqref{F2.15} and \eqref{R.R.1} that
\begin{align}
&F(q^{\frac{1}{2}},q^{a-\frac{1}{2}},q^{a+\frac{1}{2}};q)
=(-q;q)_\infty\sum_{j,k\ge 0}\frac{q^{\frac{3}{2}j^2+jk+\frac{3}{2}k^2-\frac{1}{2}j+\frac{1}{2}k+a(j+k)}}{(q^2;q^2)_j(q^2;q^2)_k}
\notag
\\
&=(-q;q)_\infty\sum_{m\ge 0}q^{\frac{1}{2}m^2+(a+\frac{1}{2})m}\sum_{j+k=m}\frac{q^{j^2-j+k^2}}{(q^2;q^2)_j(q^2;q^2)_k}
\label{sym-sum}
\\
&=(-q;q)_\infty\sum_{m\ge 0}q^{\frac{1}{2}m^2+(a+\frac{1}{2})m}\cdot \frac{q^{\frac{1}{2}(m^2-m)}}{(q;q)_m}\quad \text{(by Lemma \ref{lem-13})} \notag\\
&=(-q;q)\sum_{m\ge 0}\frac{q^{m^2+am}}{(q;q)_m}
=\frac{(-q;q)_\infty}{(q^{a+1},q^{4-a};q^5)_\infty}.\nonumber
\end{align}
This proves \eqref{table2.15.1} and \eqref{table2.15.3}.

Note that the inner sum in \eqref{sym-sum} does not change when $j$ and $k$ are interchanged. It follows that
\begin{align*}
  F(q^{\frac{1}{2}},q^{a-\frac{1}{2}},q^{a+\frac{1}{2}};q)=F(q^{\frac{1}{2}},q^{a+\frac{1}{2}},q^{a-\frac{1}{2}};q)
\end{align*}
From this equality and \eqref{table2.15.1} and \eqref{table2.15.3} we obtain \eqref{table2.15.2} and \eqref{table2.15.4}.
\end{proof}

\section{Multi-sum generalizations}\label{sec-general}
In this section, we shall provide some multi-sum Rogers--Ramanujan type identities analogous to \eqref{AG} that generalize some identities in Section \ref{sec-exam}. Throughout this section, for the summation indices $n_1,n_2,\dots,n_k$ we define
\begin{align}
    N_i:=n_i+n_{i+1}+\cdots+n_k, \quad 1\leq i\leq k, \quad \text{and} \quad N_i=0 \quad \text{for} \quad i>k.
\end{align}

\subsection{Generalization using Bailey pairs}
As is well-known (see e.g. \cite{Warnaar-survey}), we can use the theory of Bailey pairs to establish some general multi-sum identities. Here we shall use Bailey pairs to generalize the identities in Examples 12 and 13.

\subsubsection{Generalization of Theorem \ref{thm-12}}
The following result was discovered by Warnaar (unpublished). For the convenience of the reader, we shall provide a proof here.
\begin{theorem}\label{thm-Warnaar}
For $k\ge 2, 1\le i \le k$, we have
\begin{align}\label{eq-Warnaar-general}
&\sum_{n_1,\dots,n_k\geq 0} \frac{q^{\frac{1}{2}(N_1^2+\dots+N_k^2)+N_i+N_{i+2}+N_{i+4}+\cdots}}{(q;q)_{n_1}\cdots (q;q)_{n_{k-1}}(q^{2};q^{2})_{n_k}} \nonumber \\
&=\frac{(-q^\frac{1}{2};q)_\infty(q^{i/2},q^{(2k-i+3)/2},q^{(2k+3)/2};q^{(2k+3)/2})_\infty}{(q;q)_\infty}.
\end{align}
\end{theorem}
\begin{proof}
(1) First, we prove the theorem for the case $2\leq i\leq k$.

After summing over $n_1$ first, we have
\begin{align}\label{War-start}
    &\sum_{n_1,\dots,n_k\geq 0} \frac{q^{\frac{1}{2}(N_1^2+\dots+N_k^2)+N_i+N_{i+2}+N_{i+4}+\cdots}}{(q;q)_{n_1}\cdots (q;q)_{n_{k-1}}(q^{2};q^{2})_{n_k}}
    \nonumber \\
    &=\sum_{n_2,\dots,n_k\geq 0}(-q^{\frac{1}{2}+N_2};q)_\infty\frac{q^{\frac{1}{2}(2N_2^2+N_3^2+\cdots+N_k^2)+N_i+N_{i+2}+\cdots}}{(q;q)_{n_2}\cdots(q;q)_{n_{k-1}}(q^2;q^2)_{n_k}}
    \nonumber \\
    &=(-q^{\frac{1}{2}};q)_\infty\sum_{n_2,\dots,n_k\ge 0}\frac{q^{N_2^2+\frac{1}{2}(N_3^2+\cdots+N_{k}^2)+N_i+N_{i+2}+\cdots}}{(-q^\frac{1}{2};q)_{N_2}(q;q)_{n_2}\cdots(q;q)_{n_{k-1}}(q^2;q^2)_{n_k}}.
\end{align}
It suffices to evaluate the multi sum in the right side.

Applying \eqref{BL3} to the  \eqref{B-G(1.1)} Bailey pair, we obtain the Bailey pair:
\begin{equation}\label{add-BP-1}
\begin{split}
    &\alpha_n'(q;q)=(-1)^n\frac{(-q^{\frac{1}{2}};q)_n}{(-q^{\frac{3}{2}};q)_n}\frac{q^{\frac{1}{2}n^2+n}q^{\frac{3}{2}\binom{n+1}{2}}(q^{-n}-q^{n+1})}{1-q}, \\
    &\beta_n'(q;q)=\sum_{n_1=0}^n \frac{q^{\frac{1}{2}n_1^2+n_1}}{(-q^{\frac{3}{2}};q)_n(q;q)_{n-n_1}(q^2;q^2)_{n_1}}.
\end{split}
\end{equation}
Applying \eqref{BL5} to \eqref{add-BP-1} we obtain the Bailey pair:
\begin{equation}
\begin{split}
    &\alpha_n^{(2)}(q;q)=(-1)^n\frac{q^{(\frac{3}{2}+2)\binom{n+1}{2}}(q^{-n}-q^{n+1})}{1-q},  \\
    &\beta_n^{(2)}(q;q)=\sum_{n_2=0}^n\sum_{n_1=0}^{n_2}\frac{q^{\frac{1}{2}(n_2^2+n_1^2)+n_1}}{(-q^{\frac{1}{2}};q)_{n}(q;q)_{n-n_2}(q;q)_{n_2-n_1}(q^2;q^2)_{n_1}}.
\end{split}
\end{equation}
Repeating this process $r$ times, i.e., applying \eqref{BL5}\eqref{BL3} $r$ times to the  \eqref{B-G(1.1)} Bailey pair, we obtain the Bailey pair
\begin{equation}
\begin{split}
    &\alpha_n^{(2r)}(q;q)=(-1)^n\frac{q^{(\frac{3}{2}+2r)\binom{n+1}{2}}(q^{-n}-q^{n+1})}{1-q}, \\
    &\beta_n^{(2r)}(q;q)=\sum_{n_{2r}=0}^n\cdots\sum_{n_1=0}^{n_2}\frac{q^{\frac{1}{2}(n_{2r}^2+\cdots+n_1^2)+n_{2r-1}+\cdots+n_1}}{(-q^\frac{1}{2};q)_{n}(q;q)_{n-n_{2r}}\cdots(q;q)_{n_2-n_1}(q^2;q^2)_{n_1}}.\label{war-bailey-even}
\end{split}
\end{equation}
Applying \eqref{BL5} to the the \eqref{G(2)} Bailey pair, we obtain the Bailey pair:
\begin{equation}
\begin{split}
    \alpha_n^{(1)}(q;q)=(-1)^n\frac{q^{(\frac{3}{2}+1)\binom{n+1}{2}}(q^{-n}-q^{n+1})}{1-q},   \\
    \beta_n^{(1)}(1;q)=\sum_{n_1=0}^n\frac{q^{\frac{1}{2}n_1^2}}{(-q^{\frac{1}{2}};q)_{n}(q;q)_{n-n_1}(q^2;q^2)_{n_1}}.
\end{split}
\end{equation}
Applying \eqref{BL5}\eqref{BL3} $r$ times to this Bailey pair we obtain the Bailey pair:
\begin{equation}
\begin{split}
    &\alpha_n^{(2r+1)}(q;q)=(-1)^n\frac{q^{(\frac{3}{2}+1+2r)\binom{n+1}{2}}(q^{-n}-q^{n+1})}{1-q},   \\
    &\beta_n^{(2r+1)}(q;q)=\sum_{n_{2r+1}=0}^n\cdots\sum_{n_1=0}^{n_2}\frac{q^{\frac{1}{2}(n_{2r+1}^2+\cdots+n_1^2)+n_{2r}+\cdots+n_2}}{(-q^\frac{1}{2};q)_n(q;q)_{n-n_{2r+1}}\cdots(q;q)_{n_2-n_1}(q^2;q^2)_{n_1}}.\label{war-bailey-odd}
\end{split}
\end{equation}
Combining \eqref{war-bailey-even} and \eqref{war-bailey-odd} we obtain the Bailey pair for $0 \le r \le k-2$:
\begin{equation}\label{War-key-Baileypair}
\begin{split}
    &\alpha_n^{(r)}(q;q)=(-1)^n\frac{q^{(\frac{3}{2}+r)\binom{n+1}{2}}(q^{-n}-q^{n+1})}{1-q},  \\
    &\beta_n^{(r)}(q;q)=\sum_{n_r=0}^n\cdots\sum_{n_1=0}^{n_2}\frac{q^{\frac{1}{2}(n_r^2+\cdots+n_1^2)+n_{r-1}+n_{r-3}+\cdots}}{(-q^\frac{1}{2};q)_n(q;q)_{n-n_r}\cdots(q;q)_{n_2-n_1}(q^2;q^2)_{n_1}}.
\end{split}
\end{equation}
Setting $u=q^{\frac{3}{2}+r}$ in Lemma \ref{lem-BP-mod}, we obtain another Bailey pair:
\begin{equation}
\begin{split}
    &\alpha_0'(1;q)=1,\quad \alpha_n'(1;q)=(-1)^nq^{(\frac{3}{2}+r)\binom{n}{2}}(1+q^{(\frac{3}{2}+r)n}),   \\    &\beta_n'(1;q)=\sum_{n_r=0}^n\cdots\sum_{n_1=0}^{n_2}\frac{q^{n+\frac{1}{2}(n_r^2+\cdots+n_1^2)+n_{r-1}+n_{r-3}+\cdots}}{(-q^\frac{1}{2};q)_n(q;q)_{n-n_r}\cdots(q;q)_{n_2-n_1}(q^2;q^2)_{n_1}}.
\end{split}
\end{equation}

Applying \eqref{BL3} $k-r-2$ times and obtain the Bailey pair:
\begin{equation}
\begin{split}
    &\widetilde{\alpha}_0^{(k-r-2)}(1;q)=1, \quad \widetilde{\alpha}_n^{(k-r-2)}(1;q)=(-1)^nq^{\frac{(k-r-2)n^2}{2}}q^{(\frac{3}{2}+r)\binom{n}{2}}(1+q^{(\frac{3}{2}+r)n}),   \\
    &\widetilde{\beta}_n^{(k-r-2)}(1;q)=\sum_{n_{k-2}=0}^n\cdots\sum_{n_1=0}^{n_2}\frac{q^{\frac{1}{2}(n_{k-2}^2+\cdots+n_1^2)+n_{r+1}+n_{r-1}+\cdots}}{(-q^\frac{1}{2};q)_n(q;q)_{n-n_{k-2}}\cdots(q;q)_{n_2-n_1}(q^2;q^2)_{n_1}}.
\end{split}
\end{equation}
Using \eqref{cor1} with the above Bailey pair, we have
\begin{align}\label{War-key}
    &\sum_{n_{k-1}\ge 0}q^{n_{k-1}^2}\sum_{n_{k-2}=0}^{n_{k-1}}\cdots\sum_{n_1=0}^{n_2}\frac{q^{\frac{1}{2}(n_{k-2}^2+\cdots+n_1^2)+n_{r+1}+n_{r-1}+\cdots}}{(-q^\frac{1}{2};q)_{n_{k-1}}(q;q)_{n_{k-1}-n_{k-2}}\cdots(q;q)_{n_2-n_1}(q^2;q^2)_{n_1}}
    \nonumber \\
    &=\frac{1}{(q;q)_\infty}\Big(1+\sum_{n\ge 1}(-1)^nq^{n^2}q^{\frac{(k-r-2)n^2}{2}}q^{(\frac{3}{2}+r)\binom{n}{2}}(1+q^{(\frac{3}{2}+r)n})\Big)
    \nonumber \\
    &=\frac{1}{(q;q)_\infty}\sum_{n=-\infty}^\infty(-1)^nq^{\frac{(k-r)}{2}n^2}q^{(\frac{3}{2}+r)\binom{n}{2}}=\frac{1}{(q;q)_\infty}\sum_{n=-\infty}^\infty(-1)^nq^{\frac{(k-r)}{2}n}q^{(\frac{3}{2}+k)\binom{n}{2}} \nonumber \\
    &=\frac{(q^{\frac{k-r}{2}},q^{\frac{k+r+3}{2}},q^{\frac{3}{2}+k};q^{\frac{3}{2}+k})}{(q;q)_\infty}.
\end{align}
Setting $r=k-i$ and substituting \eqref{War-key} into \eqref{War-start}, we prove \eqref{eq-Warnaar-general} in the case $2\leq i\leq k$.

(2) Now we prove the theorem for the case $i=1$.

After summing over $n_1$ first, we have
\begin{align}\label{war-1-start}
    &\sum_{n_1,\dots,n_k\geq 0} \frac{q^{\frac{1}{2}(N_1^2+\dots+N_k^2)+N_1+N_{3}+N_{5}+\cdots}}{(q;q)_{n_1}\cdots (q;q)_{n_{k-1}}(q^{2};q^{2})_{n_k}} \nonumber \\
    &=\sum_{n_2, \dots ,n_k\ge 0}(-q^{\frac{3}{2}+N_2};q)_\infty\frac{q^{N_2^2+N_2+\frac{1}{2}(N_3^2+\cdots+N_k^2)+N_3+N_5+\cdots}}{(q;q)_{n_2}\cdots(q;q)_{n_{k-1}}(q^2;q^2)_{n_k}} \nonumber \\
    &=(-q^{\frac{3}{2}};q)_\infty\sum_{n_2,\dots,n_k\ge 0}\frac{q^{N_2^2+N_2+\frac{1}{2}(N_3^2+\cdots+N_k^2)+N_3+N_5+\cdots}}{(-q^{\frac{3}{2}};q)_{N_2}(q;q)_{n_2}\cdots(q;q)_{n_{k-1}}(q^2;q^2)_{n_k}}.
\end{align}

When $k=2$, the theorem follows from \eqref{war-1-start} and substitution of the \eqref{G(2)} Bailey pair into \eqref{cor1}.

From now on we assume that $k\geq 3$. Applying \eqref{BL3} to the Bailey pair \eqref{War-key-Baileypair}, we obtain the Bailey pair:
\begin{equation}
\begin{split}
    &\alpha_n(q;q)=(-1)^n\frac{q^{(\frac{3}{2}+r+1)\binom{n+1}{2}}(q^{-\frac{n}{2}}-q^{\frac{n+1}{2}})}{1-q^{\frac{1}{2}}},\nonumber\\
&\beta_n(q;q)=\sum_{n_{r+1}=0}^{n}\cdots\sum_{n_1=0}^{n_2}
    \frac{q^{\frac{1}{2}(n_{r+1}^2+\cdots+n_1^2)+n_{r+1}+n_{r-1}+\cdots}}{(-q^{\frac{3}{2}};q)_{n}(q;q)_{n-n_{r+1}}\cdots(q;q)_{n_2-n_1}(q^2;q^2)_{n_1}}.
\end{split}
\end{equation}
Using \eqref{cor1} with this Bailey pair, we have
\begin{align}\label{war-1-key}
    &\sum_{n_{r+2}\ge 0}q^{n_{r+2}^2+n_{r+2}}\sum_{n_{r+1}=0}^{n_{r+2}}\cdots\sum_{n_1=0}^{n_2}
    \frac{q^{\frac{1}{2}(n_{r+1}^2+\cdots+n_1^2)+n_{r+1}+n_{r-1}+\cdots}}{(-q^{\frac{3}{2}};q)_{n_{r+2}}(q;q)_{n_{r+2}-n_{r+1}}\cdots(q;q)_{n_2-n_1}(q^2;q^2)_{n_1}} \nonumber \\
    &=\frac{1}{(q^2;q)_\infty}\sum_{n\ge 0}(-1)^nq^{n^2+n}\frac{q^{(\frac{3}{2}+r+1)\binom{n+1}{2}}(q^{-\frac{n}{2}}-q^{\frac{n+1}{2}})}{1-q^{\frac{1}{2}}} \nonumber  \\
    &= \frac{1+q^{\frac{1}{2}}}{(q;q)_\infty}\sum_{n=-\infty}^\infty(-1)^nq^{(\frac{3}{2}+r+3)\binom{n+1}{2}-\frac{1}{2}n}
    =\frac{1+q^{\frac{1}{2}}}{(q;q)_\infty}\sum_{n=-\infty}^\infty(-1)^nq^{(r+4)n}q^{(\frac{3}{2}+r+3)\binom{n}{2}} \nonumber \\
    &=\frac{(1+q^{\frac{1}{2}})(q^{\frac{1}{2}},q^{r+4},q^{\frac{3}{2}+r+3};q^{\frac{3}{2}+r+3})_\infty}{(q;q)_\infty}.
\end{align}
Setting $r=k-3$ and substituting \eqref{war-1-key} into \eqref{war-1-start}, we prove \eqref{eq-Warnaar-general} in this last case.
\end{proof}

\subsubsection{Generalizations of Theorem \ref{thm-13}}
As mentioned in the introduction, the key identities for generalizing identities in Example 13 are Theorems \ref{thm-gen-13-original} and \ref{thm-gen-13-odd}. Now we present proofs and consequences of them.
\begin{proof}[Proof of Theorem \ref{thm-gen-13-original}]
Applying \eqref{Bailey lemma-infty} to the  \eqref{B-G(1.1)} Bailey pair, we obtain the new Bailey pair:
\begin{equation}
\begin{split}
    &\alpha_{n}^{(1)}(q;q)=(-1)^n\frac{^{(\frac{3}{2}+2)\binom{n+1}{2}}(q^{-n}-q^{n+1})}{1-q},  \\
    &\beta_{n}^{(1)}(q;q)=\sum_{n_1=0}^{n}\frac{q^{n_1^2+n_1}}{(q;q)_{n-n_1}(-q^{1/2};q)_{n_1}(q^2;q^2)_{n_1}}.
\end{split}
\end{equation}
We repeat this process $k-i$ times to obtain the Bailey pair:
\begin{equation}\label{add-13-BP-1}
\begin{split}
    &\alpha_n^{(k-i)}(q;q)=(-1)^n\frac{q^{(\frac{3}{2}+2(k-i))\binom{n+1}{2}}(q^{-n}-q^{n+1})}{1-q},  \\
    &\beta_{n}^{(k-i)}(q;q)=\sum_{n_{k-i}=0}^{n}\cdots \sum_{n_1=0}^{n_2}\frac{q^{n_{k-i}^2+n_{k-i}+...+n_1^2+n_1}}{(q;q)_{n-n_{k-i}}...(q;q)_{n_{2}-n_1}(-q^{1/2};q)_{n_1}(q^2;q^2)_{n_1}}.
\end{split}
\end{equation}

Setting $u=q^{\frac{3}{2}+2(k-i)}$ in Lemma \ref{lem-BP-mod}, we obtain from \eqref{add-13-BP-1} another Bailey pair:
\begin{equation}\label{add-13-BP-2}
\begin{split}
    &\alpha_1(1;q)=1,\quad \alpha_n(1;q)=(-1)^nq^{(\frac{3}{2}+2(k-i))\binom{n}{2}}(1+q^{(\frac{3}{2}+2(k-i))n}),  \\
    &\beta_n(1;q)=\sum_{n_{k-i}=0}^{n}...\sum_{n_1=0}^{n_2}\frac{q^{n+n_{k-i}^2+n_{k-i}+...+n_1^2+n_1}}{(q;q)_{n-n_{k-i}}...(q;q)_{n_{2}-n_1}(-q^{1/2};q)_{n_1}(q^2;q^2)_{n_1}}.
\end{split}
\end{equation}

Applying \eqref{Bailey lemma-infty} $i-1$ times to \eqref{add-13-BP-2}, we  obtain the Bailey pair:
\begin{equation}\label{BP-ik-1}
\begin{split}
    &\alpha_1^{(i-1)}(1;q)=1,\quad \alpha_n^{(i-1)}(1;q)=(-1)^nq^{(i-1)n^2}q^{(\frac{3}{2}+2(k-i))\binom{n}{2}}(1+q^{(\frac{3}{2}+2(k-i))n}), \\
    &\beta_{n}^{(i-1)}(1;q)=\sum_{n_{k-1}=0}^{n}\cdots \sum_{n_1=0}^{n_2}\frac{q^{n_{k-1}^2+\cdots+n_{1}^2+n_{k-i+1}+\cdots+n_{1}}}{(q;q)_{n-n_{k-1}}\cdots(q;q)_{n_2-n_1}(-q^{1/2};q)_{n_1}(q^2;q^2)_{n_1}}.
\end{split}
\end{equation}
Note that the above process is valid only for $1\leq i\leq k$. But in fact, \eqref{BP-ik-1} is also a Bailey pair when $i=k+1$. This case can be obtained by applying \eqref{Bailey lemma-infty} $k-1$ times to the \eqref{B-G(1)} Bailey pair.

Using \eqref{cor1} with the Bailey pair \eqref{BP-ik-1}, we have for $1\leq i\leq k+1$ that
\begin{align*}
    &\sum_{n_{k}\ge 0}q^{n_{k}^2} \sum_{n_{k-1}=0}^{n_k}\cdots\sum_{n_1=0}^{n_2}\frac{q^{n_{k-1}^2+\cdots+n_{1}^2+n_{k-i+1}+\cdots+n_{1}}}{(q;q)_{n_k-n_{k-1}}..(q;q)_{n_2-n_1}(-q^{1/2};q)_{n_1}(q^2;q^2)_{n_1}}\nonumber \\
    &=
    \frac{1}{(q;q)_\infty}\Big(1+\sum_{n\ge 1}(-1)^nq^{n^2}q^{(i-1)n^2}q^{(\frac{3}{2}+2(k-i))\binom{n}{2}}(1+q^{(\frac{3}{2}+2(k-i))n})\Big) \nonumber \\
    &=\frac{1}{(q;q)_\infty}\sum_{n=-\infty}^\infty (-1)^nq^{in^2}q^{(\frac{3}{2}+2(k-i))\binom{n}{2}} =\frac{1}{(q;q)_\infty}\sum_{n=-\infty}^{\infty}(-1)^nq^{in}q^{(\frac{3}{2}+2k)\binom{n}{2}} \nonumber \\
    &=\frac{(q^i,q^{\frac{3}{2}+2k-i},q^{\frac{3}{2}+2k};q^{\frac{3}{2}+2k})_\infty}{(q;q)_\infty}. \qedhere
\end{align*}
\end{proof}
As a consequence of Theorem \ref{thm-gen-13-original}, we have the following identity which generalizes the first three identities in Theorem \ref{thm-13}.
\begin{corollary}\label{cor-gen-13}
For $k\geq 1$, $1 \le i \le k+1$, we have
\begin{align}
   & \sum_{m,n_1,\dots,n_k\ge 0}\frac{q^{\frac{1}{2}m^2+mn_k+(N_1^2+N_2^2+\cdots+N_k^2)+N_i+N_{i+1}+\cdots+N_k}}{(q;q)_{m}(q;q)_{n_1}\cdots(q;q)_{n_{k-1}}(q^2;q^2)_{n_k}} \nonumber \\
&    =
    \frac{(-q^{\frac{1}{2}};q)_\infty(q^{i},q^{\frac{3}{2}+2k-i},q^{\frac{3}{2}+2k};q^{\frac{3}{2}+2k})_\infty}{(q;q)_\infty}. \label{ex13-general}
\end{align}
\end{corollary}
\begin{proof}
After summing over $m$ first, we deduce that
\begin{align*}
   & \sum_{m,n_1,\dots,n_k\ge 0}\frac{q^{\frac{1}{2}m^2+mn_k+(N_1^2+N_2^2+\cdots+N_k^2)+N_i+N_{i+1}+\cdots+N_k}}{(q;q)_{m}(q;q)_{n_1}\cdots(q;q)_{n_{k-1}}(q^2;q^2)_{n_k}} \nonumber \\
&    =\sum_{n_1,\dots,n_k\ge 0}(-q^{\frac{1}{2}+n_k};q)_\infty \frac{q^{N_1^2+N_2^2+\cdots+N_k^2+N_i+N_{i+1}+\cdots+N_k}}{(q;q)_{n_1}\cdots(q;q)_{n_{k-1}}(q^2;q^2)_{n_k}} \nonumber \\
&    =(-q^{\frac{1}{2}};q)_\infty \sum_{n_1,\dots,n_k\ge 0} \frac{q^{N_1^2+N_2^2+\cdots+N_k^2+N_i+N_{i+1}+\cdots+N_k}}{(q;q)_{n_1}\cdots(q;q)_{n_{k-1}}(-q^{\frac{1}{2}};q)_{n_k}(q^2;q^2)_{n_k}}.
\end{align*}
The desired identity then follows from Theorem \ref{thm-gen-13-original}.
\end{proof}
Setting $(k,i)=(2,3)$, $(2,2)$ and $(2,1)$ in \eqref{ex13-general} and replacing $q$ by $q^2$, we obtain \eqref{table2.13.1}--\eqref{table2.13.3}, respectively.

\begin{proof}[Proof of Theorem \ref{thm-gen-13-odd}]
Applying \eqref{Bailey lemma-infty} $k-1$ times to the \eqref{G(2)} Bailey pair, we obtain the Bailey pair:
\begin{equation}
\begin{split}
    &\alpha_n^{(k-1)}(q;q)=(-1)^n\frac{q^{(\frac{3}{2}+2(k-1))\binom{n+1}{2}}(q^{-\frac{n}{2}}-q^{\frac{n+1}{2}})}{1-q^\frac{1}{2}},
    \nonumber \\
    &\beta_n^{(k-1)}(q;q)=\sum_{n_{k-1}=0}^n \cdots \sum_{n_1=0}^{n_2}
    \frac{q^{n_{k-1}^2+n_{k-1}+\cdots+n_1^2+n_1}}{(q;q)_{n-n_{k-1}}\cdots (q;q)_{n_2-n_1}(q^2;q^2)_{n_1}(-q^{\frac{3}{2}};q)_{n_1}}.
\end{split}
\end{equation}
Using \eqref{cor1} with the above Bailey pair, we have
\begin{align*}
    &\sum_{n_k\ge 0}q^{n_k^2+n_k}\sum_{n_{k-1}=0}^{n_k}\cdots \sum_{n_1=0}^{n_2}
    \frac{q^{n_{k-1}^2+n_{k-1}+\dots+n_1^2+n_1}}{(q;q)_{n_k-n_{k-1}}...(q;q)_{n_2-n_1}(q^2;q^2)_{n_1}(-q^{\frac{3}{2}};q)_{n_1}}
    \nonumber \\
    &=
    \frac{1}{(q^2;q)_\infty}\sum_{n\ge 0}(-1)^nq^{n^2+n}\frac{q^{(\frac{3}{2}+2(k-1))\binom{n+1}{2}}(q^{-\frac{n}{2}}-q^{\frac{n+1}{2}})}{1-q^{\frac{1}{2}}}
    \nonumber \\
    &=\frac{1+q^{\frac{1}{2}}}{(q;q)_\infty}\sum_{n=-\infty}^\infty(-1)^nq^{(\frac{3}{2}+2k)\binom{n+1}{2}-\frac{n}{2}}=\frac{1+q^{\frac{1}{2}}}{(q;q)_\infty}\sum_{n=-\infty}^\infty(-1)^nq^{(1+2k)n+(\frac{3}{2}+2k)\binom{n}{2}}
    \nonumber\\
    &=\frac{(1+q^{\frac{1}{2}})(q^{\frac{1}{2}},q^{1+2k},q^{\frac{3}{2}+2k};q^{\frac{3}{2}+2k})_\infty}{(q;q)_\infty}. \qedhere 
\end{align*}
\end{proof}
As a consequence of Theorem \ref{thm-gen-13-odd}, we have the following result which reduces to the last identity in Theorem \ref{thm-13} when $k=2$.
\begin{corollary}\label{cor-gen-13-last}
For $k\geq 1$, we have
\begin{align}
   & \sum_{m,n_1,\dots,n_k\ge 0}\frac{q^{\frac{1}{2}m^2+mn_k+(N_1^2+N_2^2+\cdots+N_k^2)+m+N_1+\cdots+N_k}}{(q;q)_{m}(q;q)_{n_1}\cdots(q;q)_{n_{k-1}}(q^2;q^2)_{n_k}} \nonumber \\
&    =
    \frac{(-q^{\frac{1}{2}};q)_\infty(q^{\frac{1}{2}},q^{1+2k},q^{\frac{3}{2}+2k};q^{\frac{3}{2}+2k})_\infty}{(q;q)_\infty}. \label{ex13-general-0.5}
\end{align}
\end{corollary}
\begin{proof}
After summing over $m$ first, we have
\begin{align}\label{general-13-odd-2-start}
   & \sum_{m,n_1,\dots,n_k\ge 0}\frac{q^{\frac{1}{2}m^2+mn_k+(N_1^2+N_2^2+\cdots+N_k^2)+m+N_1+\cdots+N_k}}{(q;q)_{m}(q;q)_{n_1}\cdots(q;q)_{n_{k-1}}(q^2;q^2)_{n_k}} \nonumber \\
&    =\sum_{n_1,\dots,n_k\geq 0} (-q^{n_k+\frac{3}{2}};q)_\infty \frac{q^{N_1^2+N_2^2+\cdots+N_k^2+N_1+\cdots+N_k}}{(q;q)_{m}(q;q)_{n_1}\cdots(q;q)_{n_{k-1}}(q^2;q^2)_{n_k}} \nonumber \\
&    =(-q^{\frac{1}{2}};q)_\infty \sum_{n_1,\dots,n_k\geq 0} \frac{q^{N_1^2+N_2^2+\cdots+N_k^2+N_1+\cdots+N_k}}{(q;q)_{n_1}\cdots(q;q)_{n_{k-1}}(-q^{\frac{1}{2}};q)_{n_k+1}(q^2;q^2)_{n_k}}.
\end{align}
The desired identity then follows from Theorem \ref{thm-gen-13-odd}.
\end{proof}

Corollaries \ref{cor-gen-13} and \ref{cor-gen-13-last} give a complete generalization of Theorem \ref{thm-13}.

\subsection{Generalizations from known identities}
As we have seen in Section \ref{sec-exam}, Euler's identities \eqref{euler}, Lemmas \ref{lem-12} and \ref{lem-13} can be used to reduce some rank three Nahm sums to rank two Nahm sums. For example, we used Lemma \ref{lem-12} in Examples 3, 6, 9, 10, 14, 17, 18 and 19, and we used Lemma \ref{lem-13} in Example 15 so that the triple sums are converted to double sums.

If we use Euler's identities \eqref{euler}, Lemmas \ref{lem-12} and \ref{lem-13} in another direction, we can also express a rank $r$ Nahm sum as a rank $r+1$ Nahm sum.  Hence we can get modular rank $r+1$ Nahm sums from modular rank $r$ Nahm sums. Below we discuss two sets of examples to demonstrate this process.

\subsubsection{Generalization from the Andrews-Gordon identity}
If we start with the rank $k$ Nahm sum in the Andrews--Gordon identity \eqref{AG}, we can get many higher rank Nahm sum identities. Here we list several such examples which generalize some aforementioned identities.

\begin{theorem}\label{thm-gen-5-8}
$(1)$ (Generalization of Theorem $\ref{thm-5}$.) For $k\ge 1, 1 \le i \le k+1$, we have
\begin{align}\label{example5-general}
    &\sum_{m,n_1,\dots,n_k\ge 0}\frac{q^{\binom{m+1}{2}+mn_1+2N_1^2+\cdots+2N_k^2+2N_i+\cdots+2N_k}}{(q;q)_{m}(q;q)_{n_1}(q^2;q^2)_{n_2}\cdots(q^2;q^2)_{n_k}}
    =
    \frac{(q^{2i},q^{4k+6-2i},q^{4k+6};q^{4k+6})_\infty}{(q;q)_\infty}.
\end{align}
$(2)$ (Generalization of Theorem $\ref{thm-8}$.) For $k\ge 1, 1 \le i \le k+1$, we have
\begin{align}
    \sum_{m,n_1,\dots,n_k \geq 0}\frac{q^{\binom{m+1}{2}+mn_k+2N_1^2+\cdots+2N_k^2+2N_i+\cdots+2N_k}}{(q;q)_m(q^2;q^2)_{n_1}\cdots(q^2;q^2)_{n_{k-1}}(q;q)_{n_k}}
    =
    \frac{(q^{2i},q^{4k+6-2i},q^{4k+6};q^{4k+6})_\infty}{(q;q)_\infty}.
\end{align}
\end{theorem}
Summing over $m$ first, the two identities reduce to the Andrews-Gordon identity \eqref{AG}.

\begin{theorem}\label{thm-gen-1}
(Generalization of Theorem $\ref{thm-1}$.) For $k\ge 1, 1 \le i \le k+1$, we have
\begin{align}\label{example1-general}
    &\sum_{m_1,m_2,n_1,\dots,n_k\ge 0}\frac{q^{\binom{m_1+1}{2}+m_1m_2+2\binom{m_2+1}{2}+2m_2n_1+4N_1^2+\cdots+4N_k^2+4N_i+\cdots+4N_k}}{(q;q)_{m_1}(q;q)_{m_2}(q^2;q^2)_{n_1}(q^4;q^4)_{n_2}\cdots(q^4;q^4)_{n_k}}
    \nonumber \\
    &=\frac{(q^{4i},q^{8k+12-4i},q^{8k+12};q^{8k+12})_\infty}{(q;q)_\infty}.
\end{align}
\end{theorem}
\begin{proof}
After summing over $m_1$, we have by \eqref{example5-general} that
\begin{align*}
    &\sum_{m_1,m_2,n_1,\dots,n_k\ge 0}\frac{q^{\binom{m_1+1}{2}+m_2m_1+2\binom{m_2+1}{2}+2m_2n_1+4N_1^2+\cdots+4N_k^2+4N_i+\cdots+4N_k}}{(q;q)_{m_1}(q;q)_{m_2}(q^2;q^2)_{n_1}(q^4;q^4)_{n_2}\cdots(q^4;q^4)_{n_k}}\nonumber \\
    &=\sum_{m_2,n_1,\dots,n_k\ge 0}\frac{(-q^{1+m_2};q)_\infty q^{2\binom{m_2+1}{2}+2m_2n_1+4N_1^2+\cdots+4N_k^2+4N_i+\cdots+4N_k}}{(q;q)_{m_2}(q^2;q^2)_{n_1}(q^4;q^4)_{n_2}\cdots(q^4;q^4)_{n_k}} \nonumber
    \\
    &=(-q;q)_\infty\sum_{m_2,n_1,\cdots,n_k\ge 0}\frac{q^{2\binom{m_2+1}{2}+2m_2n_1+4N_1^2+\cdots+4N_k^2+4N_i+\cdots+4N_k}}{(q^2;q^2)_{m_2}(q^2;q^2)_{n_1}(q^4;q^4)_{n_2}\cdots(q^4;q^4)_{n_k}} \nonumber
    \\
    &=\frac{(q^{4i},q^{8k+12-4i},q^{8k+12};q^{8k+12})_\infty}{(q;q)_\infty}. \qedhere
\end{align*}
\end{proof}

\begin{theorem}\label{thm-gen-6}
(Generalization of Theorem $\ref{thm-6}$.) For $k\ge 1, 1 \le i \le k+1$, we have
\begin{align} \label{ex6-general}
    &\sum_{m,n_{1,1},n_{1,2},n_2,\dots,n_k\geq 0}\frac{q^{\binom{m+1}{2}+mn_1+\binom{n_{1,1}}{2}+2N_1^2+\cdots+2N_k^2+2N_{i}+\cdots+2N_{k}}}{(q;q)_m(q;q)_{n_{1,1}}(q^2;q^2)_{n_{1,2}}(q^2;q^2)_{n_2}\cdots(q^2;q^2)_{n_k}} \nonumber \\
    &=\frac{(q^{2i},q^{4k+6-2i},q^{4k+6};q^{4k+6})_\infty}{(q;q)_\infty}
\end{align}
where $n_1:=n_{1,1}+2n_{1,2}$.
\end{theorem}
\begin{proof}
 We have by Lemma \ref{lem-12} and \eqref{example5-general} that
 \begin{align*}
    &\sum_{m,n_{1,1},n_{1,2},n_2,\dots,n_k\geq 0}\frac{q^{\binom{m+1}{2}+mn_1+\binom{n_{1,1}}{2}+2N_1^2+\cdots+2N_k^2+2N_{i}+\cdots+2N_{k}}}{(q;q)_m(q;q)_{n_{1,1}}(q^2;q^2)_{n_{1,2}}(q^2;q^2)_{n_2}\cdots(q^2;q^2)_{n_k}} \nonumber \\
     &=\sum_{m,n_1,\dots,n_k\ge 0}\frac{q^{\binom{m+1}{2}+mn_1+2N_1^2+\cdots+2N_k^2+2N_i+\cdots+2N_k}}{(q;q)_m(q^2;q^2)_{n_2}\cdots(q^2;q^2)_{n_k}}\sum_{n_{1,1}+2n_{1,2}=n_1}\frac{q^{\binom{n_{1,1}}{2}}}{(q;q)_{n_{1,1}}(q^2;q^2)_{n_{1,2}}}
     \nonumber\\
     &=\sum_{m,n_1,\dots,n_k\ge 0}\frac{q^{\binom{m+1}{2}+mn_1+2N_1^2+\cdots+2N_k^2+2N_i+\cdots+2N_k}}{(q;q)_{m}(q;q)_{n_1}(q^2;q^2)_{n_2}\cdots(q^2;q^2)_{n_k}} \nonumber \\
    &=
    \frac{(q^{2i},q^{4k+6-2i},q^{4k+6};q^{4k+6}_\infty)}{(q;q)_\infty}. \qedhere
 \end{align*}
\end{proof}

\begin{theorem}\label{thm-gen-7} 
(Generalization of Theorem $\ref{thm-7}$.) For $k\ge 1, 1 \le i \le k+1$, we have
\begin{align} \label{ex7-general}
    &\sum_{m_1,m_2,n_1,\dots,n_k\geq 0}\frac{q^{\binom{m_1+1}{2}+m_1n_1+2\binom{m_2+1}{2}+2m_2n_1+4N_1^2+\cdots+4N_k^2+4N_i+\cdots+4N_k}}{(q;q)_{m_1}(q^2;q^2)_{m_2}(q;q)_{n_1}(q^4;q^4)_{n_2}\cdots(q^4;q^4)_{n_k}} \nonumber \\
    &=\frac{(q^{4i},q^{8k+12-4i},q^{8k+12};q^{8k+12})_\infty}{(q;q)_\infty}.
\end{align}
\end{theorem}
After summing over $m_1$, this reduces to \eqref{example5-general}.

\begin{theorem}\label{thm-gen-10}
(Generalization of Theorem $\ref{thm-10}$.) For $k\ge 1, 1 \le i \le k+1$, we have
\begin{align}\label{ex10-general}
    &\sum_{m_1,m_2,n_1,\dots,n_k\ge 0}\frac{q^{\binom{m_1}{2}+\binom{m_1+2m_2+1}{2}+(m_1+2m_2)n_1+2N_1^2+\cdots+2N_k^2+2N_i+\cdots+2N_k}}{(q;q)_{m_1}(q^2;q^2)_{m_2}(q;q)_{n_1}(q^2;q^2)_{n_2}\cdots(q^2;q^2)_{n_{k}}}\nonumber \\
    &=
    \frac{(q^{2i},q^{4k+6-2i},q^{4k+6};q^{4k+6})_\infty}{(q;q)_\infty}.
\end{align}
\end{theorem}
This can be proved analogously to Theorem \ref{thm-gen-6} using Lemma \ref{lem-12} and \eqref{example5-general}.

\begin{theorem}\label{thm-gen-14-17}
$(1)$ (Generalization of Theorem $\ref{thm-14}$.) For $k\ge 1, 1 \le i \le k+1$, we have
\begin{align}
    \sum_{n_1,\dots,n_{k-1},n_{k,1},n_{k,2}\geq 0}\frac{q^{\binom{n_{k,1}}{2}+N_1^2+\cdots+N_k^2+N_i+\cdots+N_k}}{(q;q)_{n_1}\cdots(q;q)_{n_{k-1}}(q;q)_{n_{k,1}}(q^2;q^2)_{n_{k,2}}}=\frac{(q^{i},q^{2k+3-i},q^{2k+3};q^{2k+3})_\infty}{(q;q)_\infty}
\end{align}
where $n_k=n_{k,1}+2n_{k,2}$.

$(2)$ (Generalization of Theorem $\ref{thm-17}$.) For $k\ge 1, 1 \le i \le k+1$, we have
\begin{align}\label{ex17-general}
    \sum_{n_{1,1},n_{1,2},n_2,\dots,n_k}\frac{q^{\binom{n_{1.1}}{2}+N_1^2+\cdots+N_k^2+N_i+\cdots+N_k}}{(q;q)_{n_{1,1}}(q^2;q^2)_{n_{1,2}}(q;q)_{n_2}\cdots(q;q)_{n_k}}
    =\frac{(q^i,q^{2k+3-i},q^{2k+3};q^{2k+3})_\infty}{(q;q)_\infty}
\end{align}
where $n_{1}=n_{1,1}+2n_{1,2}$.
\end{theorem}
This can be proved analogously to Theorem \ref{thm-gen-6} using Lemma \ref{lem-12} and \eqref{AG}.

\begin{theorem}\label{thm-gen-15}
(Generalization of Theorem $\ref{thm-15}$.) For $k\ge 1, 1 \le i \le k+1$, we have
\begin{align}
    &\sum_{m,n_{1,1},n_{1,2},n_2,\dots,n_k\ge 0}\frac{q^{\binom{m+1}{2}+mn_{1,1}+n_{1,1}^2+n_{1,2}^2-n_{1,1}-\binom{n_1}{2}+N_1^2+\cdots+N_k^2+N_i+\cdots+N_k}}{(q;q)_m(q;q)_{n_{1,1}}(q^2;q^2)_{n_{1,2}}(q;q)_{n_2}\cdots(q;q)_{n_k}}\nonumber \\
    &=\frac{(-q;q)_\infty(q^i,q^{2k+3-i},q^{2k+3};q^{2k+3})_\infty}{(q;q)_\infty}\label{ex15-general-1}
    , \\
    &\sum_{m,n_{1,1},n_{1,2},n_2,\dots,n_k\ge 0}\frac{q^{\binom{m+1}{2}+mn_{1,1}+n_{1,1}^2+n_{1,2}^2-n_{1,2}-\binom{n_1}{2}+N_1^2+\cdots+N_k^2+N_i+\cdots+N_k}}{(q;q)_m(q;q)_{n_{1,1}}(q^2;q^2)_{n_{1,2}}(q;q)_{n_2}\cdots(q;q)_{n_k}}\nonumber \\
    &=\frac{(-q;q)_\infty(q^i,q^{2k+3-i},q^{2k+3};q^{2k+3})_\infty}{(q;q)_\infty},\label{ex15-general-2}
\end{align}
where $n_1=n_{1,1}+n_{1,2}$.
\end{theorem}
\begin{proof}
Summing over $m$ first, we have by Lemma \ref{lem-13} and \eqref{AG} that
\begin{align*}
    &\sum_{m,n_{1,1},n_{1,2},n_2,\dots,n_k\ge 0}\frac{q^{\binom{m+1}{2}+mn_{1,1}+n_{1,1}^2+n_{1,2}^2-n_{1,1}-\binom{n_1}{2}+N_1^2+\cdots+N_k^2+N_i+\cdots+N_k}}{(q;q)_m(q;q)_{n_{1,1}}(q^2;q^2)_{n_{1,2}}(q;q)_{n_2}\cdots(q;q)_{n_k}}
    \\&=
    \sum_{n_{1,1},n_{1,2},n_2,\dots,n_k\ge 0}\frac{(-q^{1+n_{1,1}};q)_\infty q^{n_{1,1}^2+n_{1,2}^2-n_{1,1}-\binom{n_1}{2}+N_1^2+\cdots+N_k^2+N_i+\cdots+N_k}}{(q;q)_{n_{1,1}}(q^2;q^2)_{n_{1,2}}(q;q)_{n_2}\cdots(q;q)_{n_k}}
    \\
    &=(-q;q)_\infty\sum_{n_1,\dots,n_k\ge 0}\frac{ q^{-\binom{n_1}{2}+N_1^2+\cdots+N_k^2+N_i+\cdots+N_k}}{(q;q)_{n_2}\cdots(q;q)_{n_k}}\sum_{n_{1,1}+n_{1,2}=n_1}\frac{q^{n_{1,1}^2+n_{1,2}^2-n_{1,1}}}{(q^2;q^2)_{n_{1,1}}(q^2;q^2)_{n_{1,2}}}
    \\
    &=(-q;q)_\infty\sum_{n_1,\dots,n_{k}\geq 0} \frac{q^{N_1^2+\cdots+N_{k}^2+N_i+\cdots +N_{k}}}{(q;q)_{n_1}(q;q)_{n_2}\cdots (q;q)_{n_{k-1}} (q;q)_{n_{k}}} \\
    &=\frac{(-q;q)_\infty(q^i,q^{2k+3-i},q^{2k+3};q^{2k+3})_\infty}{(q;q)_\infty}.
\end{align*}
This proves \eqref{ex15-general-1}. Note that the inner sum in the third line does not change after interchanging $n_{1,1}$ and $n_{1,2}$. Hence \eqref{ex15-general-2} holds as well.
\end{proof}

\subsubsection{Generalizations from two identities of Andrews}
Let us discuss one more example regarding generalizations of Theorems \ref{thm-3} and \ref{thm-9}. This time we start with the following identities of Andrews \cite[Theorem 3]{Andrews2010}:
\begin{enumerate}[(1)]
\item For $k\geq a \geq 1$ with $k\equiv a$ (mod 2),
\begin{align}\label{eq-And-1}
    &\sum_{n_1,n_2,\dots,n_{k-1}\geq 0} \frac{q^{N_1^2+N_2^2+\cdots+N_{k-1}^2+2N_a+2N_{a+2}+\cdots+2N_{k-2}}}{(q^2;q^2)_{n_1}\cdots (q^2;q^2)_{n_{k-2}}(q^2;q^2)_{n_{k-1}}} \nonumber \\
    &=\frac{(-q;q^2)_\infty(q^a,q^{2k+2-a},q^{2k+2};q^{2k+2})_\infty}{(q^2;q^2)_\infty}.
\end{align}
\item For $k\geq a\geq 1$ with $k$ odd and $a$ even,
\begin{align}\label{eq-And-2}
    &\sum_{n_1,n_2,\dots,n_{k-1}\geq 0} \frac{q^{N_1^2+N_2^2+\cdots+N_{k-1}^2+n_1+n_3+\cdots+n_{a-3}+N_{a-1}+N_{a}+\cdots+N_{k-1}}}{(q^2;q^2)_{n_1}\cdots (q^2;q^2)_{n_{k-2}}(q^2;q^2)_{n_{k-1}}} \nonumber \\
    &=\frac{(-q^2;q^2)_\infty(q^a,q^{2k+2-a},q^{2k+2};q^{2k+2})_\infty}{(q^2;q^2)_\infty}.
\end{align}
\end{enumerate}
See the works of Kur\c{s}ung\"oz \cite{Kursungoz2010} and Kim and Yee \cite{Kim-Yee} for companion identities of the remaining cases.

Replacing $k$ by $k+1$ in \eqref{eq-And-1} and \eqref{eq-And-2}, and then applying Lemma \ref{lem-12} to express $1/(q^2;q^2)_{n_k}$,  we obtain the following consequence.
\begin{theorem}
Let $n_k=n_{k,1}+2n_{k,2}$. For $k+1\geq a \geq 1$ with $k+1\equiv a$ $(\mathrm{mod}~~ 2)$, we have
\begin{align}
    &\sum_{n_1,\dots,n_{k-1},n_{k,1},n_{k,2}\geq 0} \frac{q^{2\binom{n_{k,1}}{2}+N_1^2+N_2^2+\cdots+N_{k}^2+2N_a+2N_{a+2}+\cdots+2N_{k-1}}}{(q^2;q^2)_{n_1}(q^2;q^2)_{n_2}\cdots (q^2;q^2)_{n_{k-1}}(q^2;q^2)_{n_{k,1}}(q^4;q^4)_{n_{k,2}}} \nonumber \\
     &=\frac{(-q;q^2)_\infty(q^a,q^{2k+4-a},q^{2k+4};q^{2k+4})_\infty}{(q^2;q^2)_\infty}.
\end{align}
For $k+1\geq a\geq 1$ with $k$ even and $a$ even, we have
\begin{align}
      &\sum_{n_1,\dots,n_{k-1},n_{k,1},n_{k,2}\geq 0} \frac{q^{2\binom{n_{k,1}}{2}+N_1^2+N_2^2+\cdots+N_{k}^2+n_1+n_3+\cdots+n_{a-3}+N_{a-1}+N_{a}+\cdots+N_{k}}}{(q^2;q^2)_{n_1}(q^2;q^2)_{n_2}\cdots (q^2;q^2)_{n_{k-2}}(q^2;q^2)_{n_{k-1}}(q^2;q^2)_{n_{k,1}}(q^4;q^4)_{n_{k,2}}} \nonumber \\
    &=\frac{(-q^2;q^2)_\infty(q^a,q^{2k+4-a},q^{2k+4};q^{2k+4})_\infty}{(q^2;q^2)_\infty}
\end{align}
\end{theorem}
Let $(k,a)=(2,1),(2,2)$ and $(2,3)$. We obtain \eqref{table2.3.5}, \eqref{table2.3.4} and \eqref{table2.3.1}, respectively.

Applying Lemma \ref{lem-12} to express $1/(q^2;q^2)_{n_1}$ in \eqref{eq-And-1} and \eqref{eq-And-2}, we obtain the following consequence.
\begin{theorem}\label{thm-exam9-gen}
Let $n_1=n_{1,1}+2n_{1,2}$. For $k\geq a \geq 1$ with $k\equiv a$ $(\mathrm{mod}~~ 2)$, we have
\begin{align}
    &\sum_{n_{1,1},n_{1,2},n_2,\dots,n_{k-1}\geq 0} \frac{q^{2\binom{n_{1,1}}{2}+N_1^2+N_2^2+\cdots+N_{k-1}^2+2N_a+2N_{a+2}+\cdots+2N_{k-2}}}{(q^2;q^2)_{n_{1,1}}(q^4;q^4)_{n_{1,2}}(q^2;q^2)_{n_2}\cdots (q^2;q^2)_{n_{k-2}}(q^2;q^2)_{n_{k-1}}} \nonumber \\
     &=\frac{(-q;q^2)_\infty(q^a,q^{2k+2-a},q^{2k+2};q^{2k+2})_\infty}{(q^2;q^2)_\infty}.
\end{align}
For $k\geq a\geq 1$ with $k$ odd and $a$ even, we have
\begin{align}
      &\sum_{n_{1,1},n_{1,2},n_2,\dots,n_{k-1}} \frac{q^{2\binom{n_{1,1}}{2}+N_1^2+N_2^2+\cdots+N_{k-1}^2+n_1+n_3+\cdots+n_{a-3}+N_{a-1}+N_{a}+\cdots+N_{k-1}}}{(q^2;q^2)_{n_{1,1}}(q^4;q^4)_{n_{1,2}}(q^2;q^2)_{n_2}\cdots (q^2;q^2)_{n_{k-2}}(q^2;q^2)_{n_{k-1}}} \nonumber \\
    &=\frac{(-q^2;q^2)_\infty(q^a,q^{2k+2-a},q^{2k+2};q^{2k+2})_\infty}{(q^2;q^2)_\infty}.
\end{align}
\end{theorem}
Let $(k,a)=(3,1), (3,2)$ and $(3,3)$. We obtain \eqref{table2.9.8}, \eqref{table2.9.6} and \eqref{table2.9.2}, respectively.

Of course, if we start with some other identities and perform the procedure as above, we can obtain more multi-sum identities. Since this process is routine and the identities we obtained are essentially equivalent to the initial identities, we will not pursue more on this.

\subsection*{Acknowledgements}
This work was supported by the National Natural Science Foundation of China (12171375).

\end{document}